%% file: manu.tex
\documentclass[hidelinks,onefignum,onetabnum]{siamart220329}

\input{ex_shared}

\ifpdf
\hypersetup{
  pdftitle={Spectral methods on a triangle and W-systems},
  pdfauthor={J. Gao and A. Iserles}
}
\fi

\begin{document}

\maketitle

% REQUIRED
\begin{abstract}
We present an overarching framework for stable spectral methods on a triangle, defined by a multivariate W-system and based on orthogonal polynomials on the triangle. Motivated by the Koornwinder orthogonal polynomials on the triangle, we introduce a Koornwinder W-system. Once discretised by this W-system, the resulting spatial differentiation matrix  is skew symmetric, affording important advantages insofar as stability and conservation of structure are concerned. We analyse the construction of the differentiation matrix and matrix vector multiplication, demonstrating optimal computational cost. Numerical convergence is illustrated through experiments with different parameter choices. As a result, our method exhibits key characteristics of a practical spectral method, inclusive of rapid convergence, fast computation and the preservation of structure of the underlying partial differential equation.
\end{abstract}

\section{Introduction}

This paper is concerned with the theoretical foundations of spectral methods for time-dependent partial differential equations (PDEs). We consider a PDE
\begin{equation}
  \label{eq:1.1}
  \frac{\partial u}{\partial t}=\mathcal{L}(t,\MM{x},u),\qquad t\geq0,\quad \MM{x}\in\Xi\subseteq\mathbb{R}^d
\end{equation}
where $\mathcal{L}$ is a (possibly nonlinear) differential operator in space variables which might also depend upon time. An initial condition $u_0\in\mathcal{H}$ is given at $t=0$ and obeys suitable boundary conditions on $\partial\Xi$. Here, $\mathcal{H}$ is a suitable separable Hilbert space, typically a Sobolev space $\mathrm{H}^p(\Xi)$ for some $p>0$.

While there are several excellent texts on spectral methods \cite{Boyd01,funaro92pad,gottlieb77nas,quarteroni94nap,Shen11,trefethen00smm}, inclusive of a monograph on spectral methods for time-dependent problems \cite{hesthaven07smt}, we prefer to follow an approach based upon \cite{iserles24ssm}. The starting point is an orthonormal basis $\Phi=\{\varphi_n\}_{n\in\mathbb{Z}_+}$ of $\mathcal{H}$. We seek to represent the solution of (\ref{eq:1.1}) in the form
\begin{displaymath}
  u(\MM{x},t)\approx u_N(\MM{x},t)=\sum_{n=0}^N a_n(t) \varphi_n(\MM{x})
\end{displaymath}
for sufficiently large $N$. The time-dependent coefficients $a_n$ are determined by the {\em Galerkin conditions\/}
\begin{equation}
  \label{eq:1.2}
  a_m'=\left\langle \sum_{n=0}^N a_n'\varphi_n,\varphi_m\right\rangle =\left\langle \mathcal{L}\!\left(t,\,\cdot\,,\sum_{n=0}^N a_n \varphi_n\right),\varphi_m\right\rangle,\qquad m=0,\ldots,N,
\end{equation}
where $\langle\,\cdot\,,\cdot\,\rangle$ is the inner product of $\mathcal{H}$ -- note that we have exploited the orthogonality of $\Phi$. This leads to an ordinary differential system, accompanied by the initial condition $\MM{a}(0)$ given by the expansion of $u_0$ in the basis $\Phi$,
\begin{displaymath}
  a_m(0)=\langle u_0,\varphi_m\rangle,\qquad m=0,\ldots,N.
\end{displaymath}
This is a considerably broader setting than in many texts on spectral methods, where $\Phi$ is either a basis of eigenfunctions of some differential operator or a set of  standard orthogonal polynomials.

A central concept in our approach is that of a {\em differentiation matrix\/} $\mathcal D$ of the basis $\Phi$, which will be introduced and motivated in the next section. It suffices to mention here that, in the presence of either zero Dirichlet or periodic or Cauchy boundary conditions, significant advantages accrue, both in terms of stability and preservation of qualitative features of (\ref{eq:1.1}), once $\mathcal D$ is skew Hermitian.

In a univariate setting the theory is fairly established. Insofar as Cauchy boundary conditions are concerned we may use the orthogonal bases introduced in \cite{iserles19oss} (and called {\em T-systems\/} in \cite{iserles24ssm}) while in the case of zero Dirichlet conditions we advocate the use of {\em W-systems\/} \cite{iserles24ssm}. For periodic boundary conditions nothing, of course, beats the classical Fourier basis. Our purpose is to construct skew-Hermitian differentiation matrix for the partial differential operator of first order for complex geometries, such as a triangle.

There exists a range of spectral and spectral element methods designed to solve PDEs independent of the temporal variable $t$ on triangles, based on bivariate orthogonal polynomials on triangles \cite{Dubiner91}. For second-order boundary value problems, sparse element stiffness matrices are obtained on triangular elements using Dubiner-type orthogonal polynomials \cite{Beuchler06}. Spectral methods based on the Dubiner polynomial on triangles have been analyzed in Sobolev spaces, providing insights into their approximate order \cite{Guo07}. Optimal spectral error estimates have been established using generalised Koornwinder polynomials on triangles \cite{Li10}. Sparse spectral methods have been developed for linear PDEs on triangles, resulting in sparse discretizations \cite{Sheehan19}. The spectral convergence properties of orthogonal polynomials on triangles have been studied extensively \cite{offner13}. In particular, the approximate order of orthogonal polynomials on triangles in Sobolev spaces has been investigated \cite{Xu17}. Additionally, the use of Jacobi polynomials to determine orthogonal polynomials on triangles has been shown to facilitate fast transforms \cite{Shen19}. Given the increasing importance of numerical simulations for time-dependent PDEs, there is a growing need for effectively stable spectral methods in this area.

In Section~2 we introduce and analyse in detail a W-system on a right triangle. We begin by exploring differentiation matrices, discussing their basic features and importance. This is followed by the definition of a W-system based on Jacobi-like weights. This makes use of the theory of orthogonal polynomials in triangles, as introduced originally by Proriol (1957) \cite{proriol57suf} but developed much further by Koornwinder (1975) \cite{koornwinder75tva}: cf.\ \cite{koornwinder21msf} for a comprehensive introduction to the subject. Section~3 is devoted in the main to the exploration of the underlying differentiation matrices and recursion relations obeyed by their entries.  Section~4 delves into questions of convergence. This is important to applications of W-systems to spectral methods and serves as a major criterion in the choice of their parameters. We accompany our exposition by a number of numerical examples including comparison with the orthogonal polynomials on the triangle. Finally, in Section~5 we investigate general Dirichlet boundary conditions. The original theory of W-systems in \cite{iserles24ssm} assumed zero Dirichlet boundary conditions. In a univariate setting this does not represent a major impediment because we can easily reduce a PDE with any Dirichlet boundary conditions to one with zero Dirichlet boundary conditions. However, this is a considerably more difficult endeavour in a triangular domain and we need a methodology from \cite{huybrechs11fh5}, based on ideas from Computer-Aided Geometric Design.

We note in passing that the combination of W-systems with arbitrary Dirichlet boundary conditions allows for the implementation of the ideas of this paper to {\em spectral element\/} methods.

A major issue is not addressed in this paper, being subject to ongoing research: numerical linear algebra considerations in the implementation of W-systems in a triangle. The efficacy of computation hinges upon the availability of fast and efficient means to effect a range of computations: multiplying a vector fir a differentiation matrix, solving linear systems and approximating functions of matrices, in particular the matrix exponential. Insofar as W-systems on the real line are concerned, using ultraspherical W-systems in a bounded interval and Laguerre W-systems in a semi-infinite interval results in semiseparable matrices of rank 1 \cite{iserles24ssm}. We recall that a matrix $A$ is semiseparable of rank $s\in\mathbb{N}$ if each submatrix of $A$ which is wholly above, or whole beneath, the main diagonal is of rank $\leq s$ \cite{vandebril08mcs}\footnote{As a matter of fact, semiseparability in \cite{vandebril08mcs} is defined only for symmetric matrices, while here it applies to skew-symmetric matrices.}. This allows for very rapid linear algebra, e.g.\ $A\MM{v}$ can be computed in $\mathcal{O}\!\left(N\right)$ operations for an $N\times N$ matrix $A$, while the linear system $A\MM{w}=\MM{b}$ can be solved in $\mathcal{O}\!\left(N^2\right)$ operations. This leads to a very efficient implementation of W-systems as an `engine' of spectral methods in one space dimension. Currently work is in progress to extend this set of ideas to triangles.

\section{W-systems on a triangle}

Let $\Omega\subseteq\mathbb{R}^2$ be a domain with Jordan boundary. We say that $\{p_{n, k}(x, y)\}_{n\in\mathbb{Z}_+, k=0,\ldots,n}$ is an {\em orthonormal polynomial system\/} in $\Omega$ once each $p_{n,k}$ is a bivariate polynomial of total degree $n$ and
\begin{displaymath}
   \langle p_{m, \ell}, p_{n, k} \rangle = \int_{\Omega} p_{m, \ell}(x, y) p_{n, k}(x, y) w(x, y)\D y \D x = \delta_{m, n} \delta_{\ell, k}
\end{displaymath}
for all $m,n\in\mathbb{Z}_+$, $\ell=0,\ldots,m$ and $k=0,\ldots,n$. Here $w$ is a weight function, nonnegative in $\Omega$, with an infinite number of points of increase and such that $\langle 1,1\rangle>0$. We assume in addition that $w(x,y)=0$ for all $(x,y)\in\partial\Omega$.

Given any orthonormal system $\Psi=\{\psi_n\}_{n\in\mathbb{Z}_+}$ which is a basis of some Hilbert space $\mathcal{H}$ and given $f\in\mathcal{H}\cap\mathrm{C}^1$, it is true that each $\psi_n'$ is itself spanned by elements of $\Psi$,
\begin{displaymath}
  \psi_n'=\sum_{m=0}^\infty \mathcal{D}_{n,m}\psi_m,\qquad n\in\mathbb{Z}_+.
\end{displaymath}
The infinite matrix $\mathcal{D}$ is called the {\em differentiation matrix\/} corresponding to $\Psi$. Once we associate $f\in\mathcal{H}$ with the vector $\hat{\MM{f}}$ of its expansion coefficients in the basis $\Psi$, then $\mathcal{D}\hat{\MM{f}}$ is the vector of expansion coefficients of $f'$ and the diagram
\begin{displaymath}
  \begin{picture}(120,90)
    \put (-5,0) {$\hat{\MM{f}}$}
    \put (92,0) {$\mathcal{D}\hat{\MM{f}}$}
    \put (-5,75) {$f$}
    \put (95,75) {$f'$}
    \thicklines
    \put (5,3) {\line(1,0){83}}
    \put (5,78) {\line(1,0){83}}
    \put (-2,12) {\line(0,1){57}}
    \put (98,12)  {\line(0,1){57}}
  \end{picture}
\end{displaymath}
is commutative.

When $f$ obeys zero (or periodic, or Cauchy) boundary conditions, the differential operator is skew-Hermitian and it is important to retain this feature under discretisation. Thus, we seek orthonormal bases of $\mathcal{H}$ such that $\mathcal{D}$ is skew Hermitian. This confers a number of important advantages:
\begin{enumerate}
\item The exponential $\ee^{t\mathcal{D}}$ is unitary, therefore (in the standard $\mathrm{L}_2$ norm) $\|\ee^{t\mathcal{D}}\|\equiv1$;
\item Because $\mathcal{D}^*=-\mathcal{D}$, it follows that the differentiation matrix corresponding to the Laplacian is $-\mathcal{D}^*\mathcal{D}$. This makes sense from the point of view of topology (the div-grad representation of the Laplacian). Moreover\ldots
\item The matrix $\mathcal{D}^2=-\mathcal{D}^*\mathcal{D}$ is negative semi-definite, therefore $\|\ee^{-t\mathcal{D}^*\mathcal{D}}\|\leq1$, $t\geq0$, while $\|\ee^{-\ii t\mathcal{D}^*\mathcal{D}}\|\equiv1$ (again, in the $\mathrm{L}_2$ norm).
\end{enumerate}
Item 1 means that the discretisation of the advection equation $\partial u/\partial t=\partial u/\partial x$ is stable, while item~3 implies that the solution of the diffusion equation is dissipative (hence stable) and the solution of the linear Schr\"odinger equation is unitary (and, again, stable). In general, once $\mathcal{D}$ is skew Hermitian, stability analysis becomes easier and it is more likely that significant qualitative features of the exact solution will be respected under discretisation.

Unfortunately, no orthonormal basis composed of polynomials can lead to a skew-Hermitian differentiation matrix: it is easy to observe this in a univariate case, when $\mathcal{D}$ is strictly lower triangular. An alternative within the setting of zero Dirichlet boundary conditions is to use {\em W-systems\/} \cite{iserles24ssm}. The underlying idea is very simple: to replace (in the present context) each $p_{n,k}$ by
\begin{displaymath}
  \varphi_{n,k}(x,y)=\sqrt{w(x,y)} p_{n,k}(x,y).
\end{displaymath}
It is trivial to verify that $\Phi=\{\varphi_{n,k}\}_{n\in\mathbb{Z}_+,\; k=0,\ldots,n}$ is an orthonormal system in $\mathrm{L}_2(\Omega)$. This, incidentally, brings another benefit: the standard $\mathrm{L}_2$ norm often (e.g.\ in the above cases of diffusion and linear Schr\"odinger equations) is the right means to `measure' the solution.

In a bivariate setting it is convenient to distinguish between $\partial/\partial x$ and $\partial/\partial y$, thus $\mathcal{D}=[\mathcal{X},\mathcal{Y}]$, where
\begin{displaymath}
  \mathcal{X}_{(m,\ell),(n,k)}=\left\langle \frac{\partial\varphi_{n,k}}{\partial x},\varphi_{m,\ell}\right\rangle\!,\qquad \mathcal{Y}_{(m,\ell),(n,k)}=\left\langle \frac{\partial\varphi_{n,k}}{\partial y},\varphi_{m,\ell}\right\rangle
\end{displaymath}
for $\ell=0,\ldots,m$, $k=0,\ldots,n$ and $m,n\in\mathbb{Z}_+$. Then the following properties can be derived.

\begin{theorem}
\label{th:1}
Let the domain $\Omega$ be simply connected and assume that the weight function vanishes along $\partial\Omega$. Then the differentiation matrices $\mathcal{X}$ and $\mathcal{Y}$ are skew-symmetric.
\end{theorem}

\begin{proof}
  It is sufficient to prove the theorem for $\mathcal{X}$ since the proof for $\mathcal{Y}$ is identical.

  Denote its positively-oriented boundary of $\Omega$ by $\Gamma$. Using the Green theorem we have
  \begin{eqnarray*}
    && \mathcal{X}_{(m, \ell),(n, k)} + \mathcal{X}_{(n, k),(m, \ell)}\\
    &=&\int_\Omega \frac{\partial \sqrt{w(x,y)} p_{n,k}(x,y)}{\partial x} \sqrt{w(x,y)}p_{m,\ell}(x,y) \D x\D y\\
    &&\mbox{}+\int_\Omega \sqrt{w(x,y)} p_{n,k}(x,y) \frac{\partial \sqrt{w(x,y)}p_{m,\ell}(x,y)}{\partial x} \D x\D y\\
    &=&\int_{\mathcal T} \frac{\partial w(x,y) p_{n,k}(x,y)p_{m,\ell}(x,y)}{\partial x}\D x\D y\\
    &=&\oint_\gamma w(x,y)p_{n,k}(x,y)p_{m,\ell}(x,y)\D s=0,
  \end{eqnarray*}
  because $w=0$ on $(x,y)\in\partial\Omega$.
\end{proof}

In the specific case of a right triangle $\mathcal{T}$ with vertices at $(0,0)$, $(1,0)$ and $(0,1)$, assuming again that the weight function vanishes along the boundary, we have
\begin{eqnarray*}
  \mathcal{X}_{(m, \ell),(n, k)}&=& \int_{\mathcal T} \frac{\partial \sqrt{w(x, y)}p_{n, k}(x, y)}{\partial x}  \sqrt{w(x, y)} p_{m, \ell}(x, y)  \D x \D y\\
  &=& \frac{1}{2}\int_{\mathcal{T}}  \frac{\partial w(x, y) }{\partial x}p_{n, k}(x, y) p_{m, \ell}(x, y) \D x \D y\\
  &&\mbox{} +\int_{\mathcal{T}} \frac{\partial p_{n, k}(x, y)}{\partial x} p_{m, \ell}(x, y) w(x, y) \D x \D y\\
  &=& \frac{1}{2}\int_{\mathcal T} \frac{\partial w(x, y) }{\partial x}p_{n, k}(x, y) p_{m, \ell}(x, y) \D x \D y.
\end{eqnarray*}
If $m\geq n$ then the second integral vanishes by orthogonality. We thus deduce that
\begin{equation}
\label{eq:2.1}
  \mathcal{X}_{(m, \ell),(n, k)} =
  \begin{cases}
    \displaystyle -\frac12 \int_{\mathcal T} \frac{\partial w(x, y)}{\partial x}p_{n, k}(x, y) p_{m, \ell}(x, y) \D x \D y, & m\leq n-1,\\[6pt]
    0, & m=n\\[6pt]
    \displaystyle \frac12 \int_{\mathcal T} \frac{\partial w(x, y)}{\partial x}p_{n, k}(x, y) p_{m, \ell}(x, y) \D x \D y, & m\geq n+1.
  \end{cases}
\end{equation}
Likewise,
\begin{equation}
\label{eq:2.2}
  \mathcal{Y}_{(m, \ell),(n, k)} =
  \begin{cases}
    \displaystyle -\frac12 \int_{\mathcal T} \frac{\partial w(x, y)}{\partial y}p_{n, k}(x, y) p_{m, \ell}(x, y) \D x \D y, & m\leq n-1,\\[6pt]
    0, & m=n\\[6pt]
    \displaystyle \frac12 \int_{\mathcal T} \frac{\partial w(x, y)}{\partial y}p_{n, k}(x, y) p_{m, \ell}(x, y) \D x \D y, & m\geq n+1.
  \end{cases}
\end{equation}

In the next section, we consider W-systems based upon Koornwinder-type weights in the right triangle and analyse in detail the derivation of their differentiation matrices.

\section{A W-system on the triangle}

We commence with a justification to restrict our attention to the right triangle $\mathcal{T}$. Consider a general triangle in $\mathbb{R}^2$ with the vertices $\MM{v}_k$, $k=0,1,2$, which are not co-linear and let $T$ be their convex hull,
\begin{displaymath}
  T=\mathrm{conv}\{\MM{v}_0,\MM{v}_1,\MM{v}_2\}.
\end{displaymath}
Then
\begin{displaymath}
  \left[
  \begin{array}{c}
         x\\y
  \end{array}
  \right]=\frac{1}{\det V} \left[
  \begin{array}{c}
         (x_0y_2-x_2y_0)-(y_2-y_0)x+(x_2-x_0)y\\
         -(x_0y_1-x_1y_0)+(y_1-y_0)x-(x_1-x_0)y
  \end{array}
  \right]\!,
\end{displaymath}
where
\begin{displaymath}
  V=\left[
  \begin{array}{ccc}
         1 & x_0 & y_0\\
         1 & x_1 & y_2\\
         1 & x_2 & y_2
  \end{array}
  \right]\!,\qquad \det V=x_0(y_2-y_1)-x_1(y_2-y_0)+x_2(y_1-y_0)\neq0,
\end{displaymath}
maps $T$ onto $\mathcal{T}$, taking vertices to vertices. Thus, it is enough to restrict our narrative to orthogonal systems in $\mathcal{T}$,  with vertices at $(0,0)$, $(1,0)$ and $(0,1)$.

While the theory of orthogonal polynomials in $\mathcal{T}$, at least for Jacobi-type weight functions \cite{yu21ops},  is known, the design of stable spectral methods for time-de\-pen\-dent partial differential equations does not allow us to work with polynomials because their differentiation matrix is strictly lower triangular.   In the univariate case an alternative is provided by {\em W-systems\/} \cite{iserles24ssm} and the main purpose of this paper is to extend this to the right-triangle $\mathcal{T}$.

Let us revisit briefly the main idea of \cite{iserles24ssm} in the case of the interval $(-1,1)$. We consider the {\em ultraspherical\/} weight function $w(x)=(1-x^2)^\alpha$ and set
\begin{displaymath}
  \varphi_n(x)=\frac{1}{\sqrt{h_n^{\alpha,\alpha}}} (1-x^2)^{\alpha/2} \mathrm{P}_n^{(\alpha,\alpha)}(x),\qquad n\in\mathbb{Z}_+,
\end{displaymath}
where $\mathrm{P}_n^{(\alpha,\beta)}$ is a Jacobi polynomial and $h_n^{(\alpha,\alpha)}$ is a normalising constant that will be specified later, so that $\Phi=\{\varphi_n\}_{n\in\mathcal{Z}_+}$ is an orthonormal system in $\mathrm{L}_2(-1,1)$ for $\alpha>-1$. Moreover, $\varphi_n(\pm1)=0$ once $\alpha>0$ and for $\alpha>1$ the differentiation matrix is skew symmetric. The set $\Phi$ is an example of a W-system and we refer the reader to \cite{iserles24ssm} for its many interesting properties and for a good strategy to choose  $\alpha>1$.

\subsection{A Koornwinder-type W-system and its differentiation matrices}

The groundwork to the theory of orthogonal polynomials in triangles has been laid by Koornwinder (1975) \cite{koornwinder75tva}. In the case of the right triangle $\mathcal{T}$ we let $\alpha,\beta,\gamma>-1$ and the weight function
\begin{displaymath}
    w(x,y)=x^\alpha y^\beta (1-x-y)^\gamma,\qquad (x,y)\in\mathcal{T}.
\end{displaymath}
The underlying orthonormal system, a counterpart of Jacobi polynomials, is
\begin{displaymath}
  p_{n,k}(x,y)=r_{n,k}(1-x)^k \mathrm{P}_{n-k}^{(\beta+\gamma+2k+1,\alpha)}(2x-1)\mathrm{P}_k^{(\gamma,\beta)}
  \!\left(\frac{2y}{1-x}-1\right)\!,
\end{displaymath}
where
\begin{eqnarray*}
  r_{n,k}&=&\frac{1}{\sqrt{h_{n-k}^{\beta+\gamma+2k+1,\alpha}h_k^{\gamma,\beta}}},\\
  h_m^{\alpha,\beta}&=&\int_0^1 (1-x)^\alpha(1+x)^\beta [\mathrm{P}_m^{(\alpha,\beta)}(x)]^2\D x=\frac{2^{1+\alpha+\beta}\G(1+\alpha+m)\G(1+\beta+m)}{m!(1+\alpha+\beta+2m)\G(1+\alpha+\beta+m)}.
\end{eqnarray*}

In order to define a W-system we need the weight function to vanish along the boundary of $\mathcal{T}$ and this is equivalent to $\alpha,\beta,\gamma>0$. The W-system based upon Koorwinder polynomials is
\begin{eqnarray}
  \label{eq:3.X}
&&\varphi_{n, k}(x, y) =x^{\frac{\alpha}{2}} y^{\frac{\beta}{2}} (1-x-y)^{\frac{\gamma}{2}} p_{n,k}(x,y)\\
  \nonumber
  &=& r_{n,k} x^{\frac{\alpha}{2}} y^{\frac{\beta}{2}} (1-x-y)^{\frac{\gamma}{2}} (1-x)^k \mathrm{P}_{n-k}^{(\beta+\gamma+2k+1,\alpha)}(2x-1)\mathrm{P}_k^{(\gamma,\beta)}\!
  \left(\frac{2y}{1-x}-1\right)
\end{eqnarray}
for $k=0,\ldots,n$ and $n\in\mathbb{Z}_+$. It is elementary to prove that
\begin{displaymath}
  \langle\varphi_{m,\ell},\varphi_{n,k}\rangle=\delta_{m,n}\delta_{k,\ell},
\end{displaymath}
where $\langle\,\cdot\,,\,\cdot\,\rangle$ is the standard $\mathrm{L}_2(\mathcal{T})$ inner product. To this end it is helpful to bear in mind that
\begin{displaymath}
  \int_{\mathcal{T}} h(x,y)\D y\D x=\int_0^1 \int_0^{1-x} h(x, y) \D y \D x = \int_0^1 \int_0^1 h(x, (1-x)t) (1-x)\D t \D x.
\end{displaymath}
We denote $\Phi=\{\varphi_{n,k}\}_{k=0,\ldots,n,\; n\in\mathbb{Z}_+}$.

We note for future use that the partial derivatives of $w(x, y)$ are
\begin{eqnarray*}
\frac{\partial w(x, y) }{\partial x} &=& x^{\alpha-1} y^\beta (1-x-y)^{\gamma-1} \left(\alpha(1-x-y) - \gamma x\right),\\
\frac{\partial w(x, y) }{\partial y} &=&  x^\alpha y^{\beta-1} (1-x-y)^{\gamma-1} \left(\beta(1-x-y)-\gamma y\right)
\end{eqnarray*}
and define the four integrals
\begin{eqnarray}
\label{eq:3.2}
 && S^{\gamma,\beta}_{\ell, k} = \int_0^1 (1-y)^{\gamma-1} y^\beta \mathrm{P}_{\ell}^{(\gamma,\beta)}
\!\left(2y-1\right)\mathrm{P}_k^{(\gamma,\beta)} \!\left(2y-1\right)\D y,\\
\label{eq:3.3}
 && \tilde{S}^{\gamma,\beta}_{\ell, k} = \int_0^1 (1-y)^\gamma y^{\beta-1} \mathrm{P}_{\ell}^{(\gamma,\beta)}
\!\left(2y-1\right)\mathrm{P}_k^{(\gamma,\beta)} \!\left(2y-1\right)\D y,\\
 \label{eq:3.4}
 &&\hspace*{2pt}I^{\alpha, \beta, \gamma}_{(m-\ell,\ell),(n-k,k)}\\
 \nonumber
 &=&\int_0^1 (1-x)^{k+\ell+\beta+\gamma+1} x^{\alpha-1}\mathrm{P}_{m-\ell}^{(\beta+\gamma+2\ell+1,\alpha)}(2x-1)\mathrm{P}_{n-k}^{(\beta+\gamma+2k+1,
 \alpha)}(2x-1) \D x,\\
 \label{eq:3.5}
 &&\hspace*{2pt}\tilde{I}^{\alpha, \beta, \gamma}_{(m-\ell,\ell),(n-k,k)}\\
 \nonumber
 &=&\int_0^1 (1-x)^{k+\ell+\beta+\gamma} x^\alpha\mathrm{P}_{m-\ell}^{(\beta+\gamma+2\ell+1,\alpha)}(2x-1)
\mathrm{P}_{n-k}^{(\beta+\gamma+2k+1,\alpha)}(2x-1) \D x.
\end{eqnarray}

We now derive explicit expressions of the elements of the underlying differentiation matrix of $\Phi$, (\ref{eq:2.1}) and (\ref{eq:2.2}). Substituting the derivative of $w(x, y)$ with respect to $x$ in (\ref{eq:2.1}), we have
\begin{eqnarray}
\label{eq:3.6}
&& \frac{2}{r_{m, \ell} r_{n, k}} \mathcal{X}_{(m, \ell),(n, k)} = \int_0^1 \int_0^{1-x} x^{\alpha-1} y^\beta (1-x-y)^{\gamma-1} \left(\alpha(1-x-y) - \gamma x\right)\nonumber\\
&&\mbox{} (1-x)^{k+\ell} \mathrm{P}_{m-\ell}^{(\beta+\gamma+2\ell+1,\alpha)}(2x-1)\mathrm{P}_\ell^{(\gamma,\beta)}
  \!\left(\frac{2y}{1-x}-1\right) \mathrm{P}_{n-k}^{(\beta+\gamma+2k+1,\alpha)}(2x-1)\nonumber\\
&&\mbox{} \mathrm{P}_k^{(\gamma,\beta)}
  \!\left(\frac{2y}{1-x}-1\right)\D y \D x\nonumber\\
&=& \int_0^1 \int_0^1 x^{\alpha-1} (1-x)^\beta t^\beta (1-x)^\gamma (1-t)^{\gamma-1} \left(\alpha(1-x)(1-t) - \gamma x\right) (1-x)^{k+\ell} \nonumber\\
&&\mbox{}\mathrm{P}_{m-\ell}^{(\beta+\gamma+2\ell+1,\alpha)}(2x-1)\mathrm{P}_\ell^{(\gamma,\beta)}
  \!\left(2t-1\right) \mathrm{P}_{n-k}^{(\beta+\gamma+2k+1,\alpha)}(2x-1)\mathrm{P}_k^{(\gamma,\beta)}
  \!\left(2t-1\right) \D t \D x\nonumber\\
&=& \int_0^1 (1-x)^{k+\ell+\beta+\gamma} x^{\alpha-1} \mathrm{P}_{m-\ell}^{(\beta+\gamma+2\ell+1,\alpha)}(2x-1) \mathrm{P}_{n-k}^{(\beta+\gamma+2k+1,\alpha)}(2x-1) \D x\nonumber\\
&&\mbox{}\int_0^1 (1-t)^{\gamma-1} t^\beta  \left(\alpha(1-x)(1-t) - \gamma x\right) \mathrm{P}_\ell^{(\gamma,\beta)}
  \!\left(2t-1\right) \mathrm{P}_k^{(\gamma,\beta)}
  \!\left(2t-1\right) \D t \nonumber\\
\hspace*{15pt}&=& \alpha h^{\gamma, \beta}_k \delta_{\ell, k} I^{\alpha, \beta, \gamma}_{(m-\ell,\ell),(n-k,k)} - \gamma S^{\gamma, \beta}_{\ell, k} \tilde{I}^{\alpha, \beta, \gamma}_{(m-\ell,\ell),(n-k,k)},
\end{eqnarray}
recalling the definition (\ref{eq:3.2}), where $\delta_{\ell, k}$ is the Kronecker delta function.

The case of $\mathcal{Y}$ is similar to $\mathcal{X}$. Thus, we derive
\begin{eqnarray}
\label{eq:3.7}
&& \frac{2}{r_{m, \ell} r_{n, k}} \mathcal{Y}_{(m, \ell), (n, k)} =\int_0^1 \int_0^{1-x} x^\alpha y^{\beta-1} (1-x-y)^{\gamma-1} \left(\beta(1-x-y)-\gamma y\right)\nonumber\\
&&\mbox{}(1-x)^{k+\ell}\mathrm{P}_{m-\ell}^{(\beta+\gamma+2\ell+1,\alpha)}(2x-1)\mathrm{P}_\ell^{(\gamma,\beta)}
  \!\left(\frac{2y}{1-x}-1\right) \mathrm{P}_{n-k}^{(\beta+\gamma+2k+1,\alpha)}(2x-1)\nonumber\\
  &&\mbox{} \mathrm{P}_k^{(\gamma,\beta)}\!\left(\frac{2y}{1-x}-1\right) \D y \D x\nonumber\\
&=& \int_0^1 \int_0^1 x^\alpha (1-x)^{\beta-1} t^{\beta-1} (1-x)^\gamma (1-t)^{\gamma-1} \left(\beta(1-t)-\gamma t\right) (1-x)^{k+\ell}\nonumber\\
&&\mbox{}\mathrm{P}_{m-\ell}^{(\beta+\gamma+2\ell+1,\alpha)}(2x-1)\mathrm{P}_\ell^{(\gamma,\beta)}
  \!\left(2t-1\right) \mathrm{P}_{n-k}^{(\beta+\gamma+2k+1,\alpha)}(2x-1) \mathrm{P}_k^{(\gamma,\beta)}\!\left(2t-1\right) \D t \D x\nonumber\\
&=& \int_0^1 (1-x)^{k+\ell+\beta+\gamma-1} x^\alpha \mathrm{P}_{m-\ell}^{(\beta+\gamma+2\ell+1,\alpha)}(2x-1) \mathrm{P}_{n-k}^{(\beta+\gamma+2k+1,\alpha)}(2x-1) \D x\nonumber\\
&&\mbox{}\int_0^1(1-t)^{\gamma-1} t^{\beta-1} \left(\beta(1-t)-\gamma t\right)\mathrm{P}_\ell^{(\gamma,\beta)}\!\left(2t-1\right) \mathrm{P}_k^{(\gamma,\beta)}\!\left(2t-1\right) \D t\nonumber\\
\hspace*{15pt}&=&\left(\beta \tilde{S}^{\gamma,\beta}_{\ell, k}-\gamma S^{\gamma,\beta}_{\ell, k}\right) \tilde{I}^{\alpha, \beta, \gamma}_{(m-\ell,\ell),(n-k,k)}
\end{eqnarray}
Note the dependence of $\mathcal{X}$ and $\mathcal{Y}$ upon the choice of $\alpha$, $\beta$ and $\gamma$.

To proceed with computation of the matrices, we may assume $m \geq n$ because of the skew symmetry of $\mathcal{X}$ and $\mathcal{Y}$. An explicit derivation of $S^{\gamma,\beta}_{\ell, k}$, $\tilde{S}^{\gamma,\beta}_{\ell, k}$, $I^{\alpha, \beta, \gamma}_{(m,\ell),(n,k)}$ and $\tilde{I}^{\alpha, \beta, \gamma}_{(m-\ell,\ell),(n-k,k)}$ features in the next subsection.

\subsection{Explicit expressions for  differentiation matrices}

In the subsection, we focus on the derivation of explicit expressions for (\ref{eq:3.2}--\ref{eq:3.5}). By virtue of skew symmetry we only need to calculate the integrals for $m \geq n$. We denote by $\pi_n$  a generic univariate polynomial of degree $n$ and by $\rho_{n, j}$  a generic constant which varies depending on the circumstances.

We commence from (\ref{eq:3.2}), which is symmetric in $\ell$ and $k$ and suppose without loss of generality that $\ell \geq k$. We divide the polynomial $\mathrm{P}_k^{(\gamma,\beta)} \!\left(2y-1\right)$ by $1-y$. The Euclidean  algorithm yields
\begin{displaymath}
  \mathrm{P}_k^{(\gamma,\beta)} \!\left(y\right) = (1-y) \pi_{k-1}(y) + \rho_{k, 0}, \qquad \rho_{k, 0} = \mathrm{P}_k^{(\gamma,\beta)} \!\left(1\right) = \frac{(\gamma+1)_k}{k!}.
\end{displaymath}
Then we have
\begin{eqnarray}
\label{eq:3.8}
S^{\gamma,\beta}_{\ell, k} &=& \int_0^1 (1-y)^{\gamma-1} y^\beta \mathrm{P}_{\ell}^{(\gamma,\beta)}\!\left(2y-1\right) \left((1-y) \pi_{k-1}(y) + \rho_{k, 0}\right)\D y \nonumber\\
&=&\frac{(\gamma+1)_k}{k!} \int_0^1 (1-y)^{\gamma-1} y^\beta \mathrm{P}_{\ell}^{(\gamma,\beta)}\!\left(2y-1\right) \D y\nonumber\\
&=& \frac{(\gamma+1)_k}{k!}\frac{\G
(\gamma)\G(\beta+\ell+1)}
{\G(\beta+\gamma+\ell+1)}\nonumber\\
&=& \frac{\G
(\gamma+k+1)}{\gamma k!}\frac{\G(\beta+\ell+1)}
{\G(\beta+\gamma+\ell+1)}, \quad \Re \gamma > 0, \quad \Re \beta > -1,
\end{eqnarray}
where we have used \cite[p.~284(16.4.2)]{Erd54} for an explicit expression for the integral. When $\ell < k$ we swap $\ell$ and $k$ in (\ref{eq:3.8}).

Similarly, subject to $\ell \geq k$, applying $y=1-u$ and (\ref{eq:3.8}) results in
\begin{eqnarray}
\label{eq:3.9}
\tilde{S}^{\gamma,\beta}_{\ell, k} &=& (-1)^{\ell+k}\int_0^1 (1-u)^{\beta-1} u^{\gamma} \mathrm{P}_{\ell}^{(\beta,\gamma)}\!\left(2u-1\right)\mathrm{P}_k^{(\beta,\gamma)}(2u-1) \D u \nonumber\\
&=& (-1)^{\ell+k}S^{\beta,\gamma}_{\ell, k} = (-1)^{\ell+k} \frac{ (\beta+1)_k}{k!}\frac{\G(\gamma+\ell+1)
\G(\beta)}{\G(\beta+\gamma+\ell+1)}\nonumber\\
&=& (-1)^{\ell+k} \frac{ \G(\beta+k+1)}{\beta k!}\frac{\G(\gamma+\ell+1)}{\G(\beta+\gamma+\ell+1)}, \qquad \Re \gamma > -1, \quad \Re \beta>0\hspace*{10pt}
\end{eqnarray}
and
\begin{displaymath}
\tilde{S}^{\gamma,\beta}_{\ell, k} = (-1)^{\ell+k} \frac{ \G(\beta+\ell+1)}{\beta \ell!}\frac{\G(\gamma+k+1)}{\G(\beta+\gamma+k+1)}, \qquad \ell <k.
\end{displaymath}
We again used \cite[p.~284(16.4.2)]{Erd54}. To make sense of the expressions $S^{\gamma, \beta}_{\ell, k}$ and $\tilde{S}^{\gamma, \beta}_{\ell, k}$we need $\Re \beta > 0$ and $\Re \gamma > 0$ and this is consistent with our assumption at the beginning of Section~3. Moreover, the relation
\begin{displaymath}
\tilde{S}^{\gamma,\beta}_{\ell, k} = (-1)^{\ell+k}S^{\beta,\gamma}_{\ell, k}
\end{displaymath}
implies that $\tilde{S}^{\beta,\beta}_{\ell, k} = (-1)^{\ell+k}S^{\beta,\beta}_{\ell, k}$ and it reduces (\ref{eq:3.7}) to
\begin{displaymath}
\frac{2}{r_{m, \ell} r_{n, k}} \mathcal{Y}_{(m, \ell), (n, k)} = [(-1)^{\ell+k}-1] \beta S^{\beta,\beta}_{\ell, k} \tilde{I}^{\alpha, \beta, \beta}_{(m-\ell,\ell),(n-k,k)}
\end{displaymath}
taking $\gamma=\beta$. It also  holds that $\mathcal{Y}_{(m, \ell), (n, k)} = 0$ for even $\ell+k$ for $\beta=\gamma$.

Once it comes to the integrals $I^{\alpha, \beta, \gamma}_{(m-\ell,\ell),(n-k,k)}$ and $\tilde{I}^{\alpha, \beta, \gamma}_{(m-\ell,\ell),(n-k,k)}$, it is trivial to check that
\begin{displaymath}
I^{\alpha, \beta, \gamma}_{(m-\ell,\ell),(n-k,k)} = I^{\alpha, \beta, \gamma}_{(n-k,k),(m-\ell,\ell)}, \qquad \tilde{I}^{\alpha, \beta, \gamma}_{(m-\ell,\ell),(n-k,k)} = \tilde{I}^{\alpha, \beta, \gamma}_{(n-k,k),(m-\ell,\ell)}
\end{displaymath}
which implies that both resulting matrices are symmetric.

Note that we only need to calculate the integral $I^{\alpha, \beta, \gamma}_{(m-\ell,\ell),(n-k,k)}$ in (\ref{eq:3.4}) for $\ell = k$ because of the presence of the Kronecker delta function $\delta_{\ell, k}$ in (\ref{eq:3.6}). Dividing $\mathrm{P}_{n-k}^{(\beta+\gamma+2k+1, \alpha)}(2x-1)$ by $x$ we have
\begin{eqnarray*}
&&\mathrm{P}_{n-k}^{(\beta+\gamma+2k+1, \alpha)}(2x-1) = x \pi_{n-k-1}(x) + \rho_{n-k,0},\\
&& \mbox{where}\hspace*{8pt}\rho_{n-k,0} = \mathrm{P}_{n-k}^{(\beta+\gamma+2k+1, \alpha)}(-1) = (-1)^{n-k} \frac{(\alpha+1)_{n-k}}{(n-k)!}.
\end{eqnarray*}
Using again \cite[p.~284(16.4.2)]{Erd54}, we finally obtain
\begin{eqnarray}
\label{eq:3.10}
&&I^{\alpha, \beta, \gamma}_{(m-k,k),(n-k,k)}\nonumber\\
&=& \int_0^1 (1-x)^{\beta+\gamma+2k+1} x^{\alpha-1}\mathrm{P}_{m-k}^{(\beta+\gamma+2k+1,\alpha)}(2x-1)\left(x \pi_{n-k-1}(x) +\rho_{n-k,0}\right) \D x\nonumber\\
&=&(-1)^{n-k} \frac{(\alpha+1)_{n-k}}{(n-k)!} \int_0^1 (1-x)^{\beta+\gamma+2k+1} x^{\alpha-1}\mathrm{P}_{m-k}^{(\beta+\gamma+2k+1,\alpha)}(2x-1) \D x\nonumber\\
&=&(-1)^{n+m} \frac{(\alpha+1)_{n-k}}{(n-k)!} \int_0^1 (1-u)^{\alpha-1} u^{\beta+\gamma+2k+1} \mathrm{P}_{m-k}^{(\alpha, \beta+\gamma+2k+1)}(2u-1) \D u\nonumber\\
&=&(-1)^{n+m} \frac{(\alpha+1)_{n-k}}{(n-k)!} \frac{\G(\beta+\gamma+m+k+2)\G(\alpha)}
{\G(\alpha+\beta+\gamma+m+k+2)}\nonumber\\
\hspace*{15pt}&=& (-1)^{n+m} \frac{\G(\alpha+n-k+1)}{(n-k)!\alpha} \frac{\G(\beta+\gamma+m+k+2)}
{\G(\alpha+\beta+\gamma+m+k+2)}.
\end{eqnarray}

To sum up,
\begin{eqnarray*}
  S_{\ell,k}^{\gamma,\beta}&=&\frac{\G (\gamma+k+1)}{\gamma k!}\frac{\G(\beta+\ell+1)} {\G(\beta+\gamma+\ell+1)},\\
  \tilde{S}_{\ell,k}^{\gamma,\beta}&=& (-1)^{\ell+k} \frac{ \G(\beta+\ell+1)}{\beta \ell!}\frac{\G(\gamma+k+1)}{\G(\beta+\gamma+k+1)},\\
  I_{(m-k,k),(n-k,k)}^{\alpha,\beta,\gamma}&=& (-1)^{n+m} \frac{\G(\alpha+n-k+1)}{(n-k)!\alpha} \frac{\G(\beta+\gamma+m+k+2)}
{\G(\alpha+\beta+\gamma+m+k+2)}.
\end{eqnarray*}
Note that, repeatedly using the standard recurrence $\G(z+1)=z\G(z)$,
\begin{eqnarray*}
  S_{0,0}^{\gamma,\beta}&=&\frac{\G(\gamma)\G(\beta+1)}{\G(\beta+\gamma+1)}=\mathrm{B}(\beta+1,\gamma),\\
  S_{\ell,0}^{\gamma,\beta}&=&\frac{\beta+\ell}{\beta+\gamma+\ell} S_{\ell-1,0}^{\gamma,\beta},\qquad \ell\geq1,\\
  S_{\ell,k}^{\gamma,\beta}&=&\frac{\gamma+k}{k} S_{\ell,k-1}^{\gamma,\beta},\qquad k\geq1;\\[8pt]
  \tilde{S}_{0,0}^{\gamma,\beta}&=&\frac{\G(\beta)\G(\gamma+1)}{\G(\beta+\gamma+1)}=\mathrm{B}(\beta,\gamma+1),\\
  \tilde{S}_{\ell,0}^{\gamma,\beta}&=&-\frac{\beta+\ell}{\ell} \tilde{S}_{\ell-1,0}^{\gamma,\beta},\qquad \ell\geq1,\\
  \tilde{S}_{\ell,k}^{\gamma,\beta}&=&-\frac{\gamma+k}{\beta+\gamma+k}\tilde{S}_{\ell,k-1}^{\gamma,\beta},\qquad k\geq1;\\[8pt]
  I_{(0,0),(0,0)}^{\alpha,\beta,\gamma}&=&\frac{\G(\alpha)\G(\beta+\gamma+2)}{\G(\alpha+\beta+\gamma+2)} =\mathrm{B}(\alpha,\beta+\gamma+2),\\
  I_{(m,0),(0,0)}^{\alpha,\beta,\gamma}&=&-\frac{\beta+\gamma+m+1}{\alpha+\beta+\gamma+m+1} I_{(m-1,0),(0,0)}^{\alpha,\beta,\gamma},\qquad m\geq1,\\
  I_{(m,0),(n,0)}^{\alpha,\beta,\gamma}&=&-\frac{\alpha+n}{n} I_{(m,0),(n-1,0)}^{\alpha,\beta,\gamma},\qquad n\geq1,\\
  I_{(m-k,k),(n-k,k)}^{\alpha,\beta,\gamma}&=&\frac{(n-k+1)(\beta+\gamma+m+k+1)}{(\alpha+n-k+1)(\alpha+\beta+\gamma+m+k+1)} \\
  &&\hspace*{40pt}\mbox{}\times I_{(m-k+1,k-1),(n-k+1,k-1)}^{\alpha,\beta,\gamma},\qquad k\geq1.
\end{eqnarray*}
Thus, all the terms can be evaluated relatively fast by recursion:  this need be done only once in the course of the computation.

We have not yet, however, evaluated $\tilde{I}_{(m-\ell,\ell),(n-k,k)}^{\alpha,\beta,\gamma}$: this is more complicated and is the theme of the next subsection. Note that we do not provide it in an explicit form, just as a recursive formula, but this is sufficient for its implementation.

\subsection{A recursive relation}

In this subsection, we focus on the recursive relation of $\tilde{I}^{\alpha, \beta, \gamma}_{(m-\ell, \ell),(n-k, k)}$ which plays an important role in the differentiation matrix. It is not enough to obtain formal expressions for the elements of the differentiation matrix: in practical computation we need a convenient means to derive them explicitly at a reasonable cost. Of course, in practical computations we need to truncate the differentiation matrix to the range $m=0, \cdots, M$, $n = 0, 1, \cdots, N$, resulting in a $(\frac12(M+1)(M+2))\times(\frac12(N+1)(N+2))$ matrix.

In this subsection we present a recursive form for rapid computation of the elements of the differentiation matrix. First we need to consider some properties of Jacobi polynomials. We recall the formula \cite[18.9.6]{dlmf}
\begin{eqnarray*}
(1-u) \mathrm{P}^{(\alpha+1, \beta)}_n\!(u) &=& a_{n, \alpha, \beta} \mathrm{P}^{(\alpha, \beta)}_n(u) - b_{n, \alpha, \beta} \mathrm{P}^{(\alpha, \beta)}_{n+1}(u),\\
a_{n, \alpha, \beta} &=& \frac{2(n+\alpha+1)}{2n+\alpha+\beta+2},\qquad b_{n, \alpha, \beta} = \frac{2(n+1)}{2n+\alpha+\beta+2}.
\end{eqnarray*}
Increasing the first parameter by $1$ in \cite[18.9.5]{dlmf} gives
\begin{eqnarray*}
\mathrm{P}^{(\alpha+2, \beta)}_n\!(u) &=& c_{n, \alpha, \beta} \mathrm{P}^{(\alpha+2, \beta)}_{n-1}(u) + d_{n, \alpha, \beta} \mathrm{P}^{(\alpha+1, \beta)}_n (u),\qquad\mbox{where}\\
c_{n, \alpha, \beta} &=& \frac{n+\beta}{n+\alpha+\beta+2},\qquad d_{n, \alpha, \beta} = \frac{2n+\alpha+\beta+2}{n+\alpha+\beta+2}.
\end{eqnarray*}
It follows easily that
\begin{eqnarray}
\label{eq:3.21}
(1-u) \mathrm{P}^{(\alpha+2, \beta)}_n\!(u) &=& c_{n, \alpha, \beta} (1-u)\mathrm{P}^{(\alpha+2, \beta)}_{n-1}(u) + d_{n, \alpha, \beta}a_{n, \alpha, \beta} \mathrm{P}^{(\alpha, \beta)}_n(u)\nonumber\\
&&\mbox{} - d_{n, \alpha, \beta} b_{n, \alpha, \beta} \mathrm{P}^{(\alpha, \beta)}_{n+1}(u).
\end{eqnarray}

\begin{lemma}
\label{lemma:2}
The integral $\tilde{I}_{(m-\ell, \ell),(n-k, k)}$ satisfies the following recursive relation:
\begin{eqnarray*}
\tilde{I}^{\alpha, \beta, \gamma}_{(m-\ell, \ell),(n-k, k)} &=& q_{m,\ell} \tilde{I}^{\alpha, \beta, \gamma}_{((m-1)-\ell,\ell),(n-k,k)} + d^{[1]}_{m, \ell} \tilde{I}^{\alpha, \beta, \gamma}_{((m-1)-(\ell-1),\ell-1),(n-k,k)}\\
&&\mbox{} - d^{[2]}_{m, \ell} \tilde{I}^{\alpha, \beta, \gamma}_{(m-(\ell-1),\ell-1),(n-k,k)},\\
&=& q_{n,k} \tilde{I}^{\alpha, \beta, \gamma}_{(m-\ell,\ell),((n-1)-k,k)} + d^{[1]}_{n, k} \tilde{I}^{\alpha, \beta, \gamma}_{(m-\ell,\ell),((n-1)-(k-1),k-1)}\\
&&\mbox{} - d^{[2]}_{n, k} \tilde{I}^{\alpha, \beta, \gamma}_{(m-\ell,\ell),(n-(k-1),k-1)},
\end{eqnarray*}
where the coefficients are
\begin{eqnarray*}
q_{m,\ell} &=& \frac{c_{m-\ell, \beta+\gamma+2(\ell-1)+1, \alpha}}{2^{\alpha+\beta+\gamma+\ell+k+1}},\\
d^{[1]}_{m, \ell} &=& \frac{d_{m-\ell, \beta+\gamma+2(\ell-1)+1, \alpha}a_{m-\ell, \beta+\gamma+2(\ell-1)+1, \alpha}}{2^{\alpha+\beta+\gamma+\ell+k+1}},\\
d^{[2]}_{m, \ell} &=&\frac{d_{m-\ell, \beta+\gamma+2(\ell-1)+1, \alpha}b_{m-\ell, \beta+\gamma+2(\ell-1)+1, \alpha}}{2^{\alpha+\beta+\gamma+\ell+k+1}}.
\end{eqnarray*}
Note that $\tilde{I}^{\alpha, \beta, \gamma}_{(m-\ell, \ell),(n-k, k)} = 0$ for negative indices $m$ and $\ell$.
\end{lemma}

\begin{proof}
  Having fixed $n$ and $k$, we allow the indices $m$ and $\ell$ to vary. Inserting (\ref{eq:3.21}) into the integral produces
\begin{eqnarray*}
&&\tilde{I}^{\alpha, \beta, \gamma}_{(m-\ell,\ell),(n-k,k)}\\
&=&
\frac{1}{2^{\alpha+\beta+\gamma+\ell+k+1}}\int_{-1}^1 (1-u)^{k+\ell+\beta+\gamma} (1+u)^\alpha  \mathrm{P}_{m-\ell}^{(\beta+\gamma+2\ell+1,\alpha)}(u)
\mathrm{P}_{n-k}^{(\beta+\gamma+2k+1,\alpha)}(u) \D u\\
&=&\frac{1}{2^{\alpha+\beta+\gamma+\ell+k+1}}\int_{-1}^1 (1-u)^{k+(\ell-1)+\beta+\gamma} (1+u)^\alpha
\mathrm{P}_{n-k}^{(\beta+\gamma+2k+1,\alpha)}(u)\\
&&\mbox{}\times\! \left[c_{m-\ell, \beta+\gamma+2(\ell-1)+1, \alpha} (1-u)\mathrm{P}^{(\beta+\gamma+2\ell+1, \alpha)}_{m-\ell-1}(u) + d_{m-\ell, \beta+\gamma+2(\ell-1)+1, \alpha}\right.\\
&&\mbox{}\times \left.a_{m-\ell, \beta+\gamma+2(\ell-1)+1, \alpha}\mathrm{P}^{(\beta+\gamma+2(\ell-1)+1, \alpha)}_{m-\ell}(u)- d_{m-\ell, \beta+\gamma+2(\ell-1)+1, \alpha} \right.\\
&&\mbox{}\times\left. b_{m-\ell, \beta+\gamma+2(\ell-1)+1, \alpha}\mathrm{P}^{(\beta+\gamma+2(\ell-1)+1, \alpha)}_{m-\ell+1}(u)\right]\!\! \D u\\
&=& \frac{c_{m-\ell, \beta+\gamma+2(\ell-1)+1, \alpha} }{2^{\alpha+\beta+\gamma+\ell+k+1}} \tilde{I}^{\alpha, \beta, \gamma}_{((m-1)-\ell,\ell),(n-k,k)} +  \frac{d_{m-\ell, \beta+\gamma+2(\ell-1)+1, \alpha}}{2^{\alpha+\beta+\gamma+\ell+k+1}} \\
&&\mbox{} \times a_{m-\ell, \beta+\gamma+2(\ell-1)+1, \alpha} \tilde{I}^{\alpha, \beta, \gamma}_{((m-1)-(\ell-1),\ell-1),(n-k,k)}-d_{m-\ell, \beta+\gamma+2(\ell-1)+1, \alpha}\\
&&\mbox{} \times \frac{b_{m-\ell, \beta+\gamma+2(\ell-1)+1, \alpha} }{2^{\alpha+\beta+\gamma+\ell+k+1}} \tilde{I}^{\alpha, \beta, \gamma}_{(m-(\ell-1),\ell-1),(n-k,k)}.
\end{eqnarray*}
The same applies once $m$ and $\ell$ are fixed. This concludes the proof of the lemma.
\end{proof}

Lemma 2 allows for recursive computation of the coefficient $\tilde{I}^{\alpha,\beta,\gamma}_{(m-\ell,\ell),(n-k,k)}$
\begin{displaymath}
  \begin{picture}(50,70)
    \thicklines
    \multiput (0,0)(4,0){16} {\line(1,0){2}}
    \multiput (0,50)(4,0){16} {\line(1,0){2}}
    \multiput (5,-5)(0,4){16} {\line(0,1){2}}
    \multiput (55,-5)(0,4){16} {\line(0,1){2}}
    \put (5,0) {\circle{4}}
    \put (5,50) {\circle{4}}
    \put (55,50) {\circle{4}}
    \put (55,0) {\circle*{5}}
    \put (7,0) {\vector(1,0){46}}
    \put (55,48) {\vector(0,-1){46}}
    \put (7,48) {\vector(1,-1){46}}
    \put (-8,60) {\small$m-1$}
    \put (51,60) {\small$m$}
    \put (-23,48) {\small$\ell-1$}
    \put (-8,-3) {\small$\ell$}
  \end{picture}
\end{displaymath}
with $k$ and $n$ fixed. In other words,  commencing from the values of $\tilde{I}^{\alpha,\beta,\gamma}_{(m,0),(n-k,k)}$ and $\tilde{I}^{\alpha,\beta,\gamma}_{(0,\ell),(n-k,k)}$, we fill in the remaining values while descending down and to the right. The pair $(0,\ell)$ corresponds to $m=\ell$ -- note that if $\ell\geq m+1$ we already know that $\tilde{I}^{\alpha,\beta,\gamma}_{(m-\ell,\ell),(n-k,k)}=0$. In other words, the matrix $\{\tilde{I}_{(m-\ell,\ell),(n-k,k)}^{\alpha,\beta,\gamma}\}$ for fixed $n$ and $k$ is weakly upper triangular and the pattern of recursion is
\begin{displaymath}
  \begin{picture}(120,135)
     \thicklines
     \multiput (0,120)(40,0){4} {\circle*{4}}
     \multiput (0,80)(40,0){1} {\circle{4}}
     \multiput (40,80)(40,0){3} {\circle*{4}}
     \multiput (0,40)(40,0){2} {\circle{4}}
     \multiput (80,40)(40,0){2} {\circle*{4}}
     \multiput (0,0)(40,0){3} {\circle{4}}
     \put (120,0) {\circle*{4}}
     \multiput (2,118)(40,-40){3} {\vector(1,-1){37}}
     \multiput (42,118)(40,-40){2} {\vector(1,-1){37}}
     \put (82,118) {\vector(1,-1){37}}
     \put (40,118) {\vector(0,-1){37}}
     \multiput (80,118)(0,-40){2} {\vector(0,-1){37}}
     \multiput (120,118)(0,-40){3} {\vector(0,-1){37}}
     \put (2,120) {\vector(1,0){37}}
     \multiput (42,80)(0,40){2} {\vector(1,0){37}}
     \multiput (82,40)(0,40){3} {\vector(1,0){37}}
     \multiput (2,80)(40,-40){3} {\vector(1,0){37}}
     \put (-30,-3) {\small $\ell=3$}
     \put (-30,37) {\small $\ell=2$}
     \put (-30,77) {\small $\ell=1$}
     \put (-30,117) {\small $\ell=0$}
     \put (-14,130) {\small$m=0$}
     \put (26,130) {\small$m=1$}
     \put (66,130) {\small$m=2$}
     \put (106,130) {\small$m=3$}
  \end{picture}
\end{displaymath}
where $\circ$ denotes a zero.

Now let's consider the initial values. Firstly, $\ell=0$. The values appearing in the first row of the figure can be computed,
\begin{eqnarray*}
\tilde{I}_{(m,0),(n-k,k)}^{\alpha, \beta, \gamma}
&=&\int_0^1 (1-x)^{k+\beta+\gamma} x^\alpha\mathrm{P}_m^{(\beta+\gamma+1,\alpha)}(2x-1)
\mathrm{P}_{n-k}^{(\beta+\gamma+2k+1,\alpha)}(2x-1) \D x\\
&=&\int_0^1 (1-x)^{\beta+\gamma+1} x^\alpha \mathrm{P}_m^{(\beta+\gamma+1,\alpha)}(2x-1)
(1-x)^{k-1}\mathrm{P}_{n-k}^{(\beta+\gamma+2k+1,\alpha)}(2x-1) \D x\\
&=& 0, \qquad m \geq n, \qquad k \geq 1,\\
\tilde{I}_{(m,0),(n,0)}^{\alpha, \beta, \gamma} &=&\int_0^1 (1-x)^{\beta+\gamma} x^\alpha\mathrm{P}_m^{(\beta+\gamma+1,\alpha)}(2x-1)
\mathrm{P}_n^{(\beta+\gamma+1,\alpha)}(2x-1) \D x\\
&=& \frac{\G(\beta+\gamma+n+1)}{n!}\frac{\G(\alpha+m+1)}
{\G(\alpha+\beta+\gamma+m+2)}, \qquad m \geq n, \qquad k =0,\\
\tilde{I}_{(0,0),(n-k,k)}^{\alpha, \beta, \gamma} &=& \int_0^1 (1-x)^{\beta+\gamma+k} x^\alpha \mathrm{P}_{n-k}^{(\beta+\gamma+2k+1,\alpha)}(2x-1) \D x\\
&=& \frac{n!\G(\beta+\gamma+k+1)\G(\alpha+n-k+1)}{(n-k)! k!\G(\alpha+\beta+\gamma+n+2)},
\end{eqnarray*}
where we have used \cite[p.~284(16.4.2)]{Erd54}.

More complicated is the diagonal $\ell=m$ and it splits into four cases.

\vspace{6pt}
\noindent {\bf{Case I:}} If  $m\leq n-1$ then we rewrite (\ref{eq:3.5}) in the form
\begin{eqnarray*}
\tilde{I}^{\alpha, \beta, \gamma}_{(0,m),(n-k,k)}&=& \int_0^1 (1-x)^{2k+\beta+\gamma+1} x^\alpha \mathrm{P}_{n-k}^{(\beta+\gamma+2k+1,\alpha)}(2x-1)\\
&&\mbox{}\times
(1-x)^{m-k-1} \mathrm{P}_{0}^{(\beta+\gamma+2m+1,\alpha)}(2x-1)
\D x = 0
\end{eqnarray*}
by orthogonality of Jacobi polynomials (shifted to $(0,1)$), since the degree of $(1-x)^{m-k-1} \mathrm{P}_{0}^{(\beta+\gamma+2m+1,\alpha)}(2x-1)$ is $m-\ell-1$, less than $n-k$.

\vspace{6pt}
\noindent {\bf{Case II:}} Once $m = n$ and $m \geq k+1$ we again use (\ref{eq:3.5}) and orthogonality,
\begin{eqnarray*}
\tilde{I}^{\alpha, \beta, \gamma}_{(0,m),(n-k,k)}&=& \int_0^1 (1-x)^{2k+\beta+\gamma+1} x^\alpha \mathrm{P}_{m-k}^{(\beta+\gamma+2k+1,\alpha)}(2x-1)\\
&&\mbox{}
(1-x)^{m-k-1} \mathrm{P}_{0}^{(\beta+\gamma+2m+1,\alpha)}(2x-1)
\D x = 0,
\end{eqnarray*}
because the degree of $(1-x)^{m-k-1} \mathrm{P}_{0}^{(\beta+\gamma+2\ell+1,\alpha)}(2x-1)$ is $m-k-1\geq0$, less than the degree of $\mathrm{P}_{m-k}$.

\vspace{6pt}
\noindent {\bf{Case III:}} Suppose that $m=n=k=\ell$. Then
\begin{eqnarray*}
  \tilde{I}_{(0,m),(0,m)}^{\alpha,\beta,\gamma}&=&\int_0^1 (1-x)^{2m+\beta+\gamma}x^\alpha [\mathrm{P}_0^{(\beta+\gamma+2m+1,\alpha)}(x)]^2\D x\\
  &=&\frac{\mathrm{\Gamma}(\alpha+1)\mathrm{\Gamma}(\beta+\gamma+2m+1)}{\mathrm{\Gamma}(\alpha+\beta+\gamma+2m+2)}=\mathrm{B}(\alpha+1,\beta+\gamma+2m+1),
\end{eqnarray*}
where $\mathrm{B}$ is the Beta function.

\vspace{6pt}
\noindent {\bf{Case IV:}} When $m\geq n+1$ we use the symmetry implicit in (\ref{eq:3.5}),
\begin{displaymath}
  \tilde{I}_{(0,m),(n-k)}^{\alpha,\beta,\gamma}=\tilde{I}_{(n-k),(0,m)}^{\alpha,\beta,\gamma}.
\end{displaymath}

To sum up, the `boundary values'  for all $\ell\leq m$ needed by the recursive algorithm are
\begin{equation}
  \label{eq:3.22}
  \tilde{I}_{(m-\ell,\ell),(n-k,k)}^{\alpha,\beta,\gamma}=
  \begin{cases}
    0, & m\leq n-1,\\[4pt]
    0, & m=n\geq k+1,\\[4pt]
    \mathrm{B}(\alpha+1,\beta+\gamma+2m+1), & m=n=k,\\[4pt]
    \tilde{I}_{(n-k,k),(m-\ell,\ell)}^{\alpha, \beta,\gamma}, & m\geq n+1.
  \end{cases}
\end{equation}

Taking the above recursive relations, the differentiation matrices can be constructed which also facilitates the next iterative algorithm with the optimal matrix vector multiplication developed in Subsection 3.5. Here we present the sparse structure of the matrix constituted of $\tilde{I}_{(m-\ell,\ell),(n-k,k)}^{\alpha,\beta,\gamma}$. Let us list some of its properties:
\begin{eqnarray*}
\tilde{I}^{\alpha, \beta, \gamma}_{(m-\ell,\ell),(n-k,k)} &=& 0, \quad m = n,\quad \ell \neq k,\\
\tilde{I}^{\alpha, \beta, \gamma}_{(m-\ell,\ell),(n-k,k)} &=& 0, \quad m > n, \quad \ell < k.
\end{eqnarray*}
Then the submatrix $\tilde{I}^{\alpha, \beta, \gamma}_{m, n}$ has the structure
\begin{eqnarray*}
\tilde{I}^{\alpha, \beta, \gamma}_{m, n} &=& \left[\begin{array}{llllllll}
\tilde{I}_{(m-0,0),(n-0,0)} & 0 &0&\cdots & 0\\
\tilde{I}_{(m-1,1),(n-0,0)} & \tilde{I}_{(m-1,1),(n-1,1)} &0&\cdots & 0\\
\tilde{I}_{(m-2,2),(n-0,0)} & \tilde{I}_{(m-2,2),(n-1,1)}& \tilde{I}_{(m-2,2),(n-2,2)} &\cdots & 0\\
\vdots\\
\tilde{I}_{(m-n,n),(n-0,0)} & \tilde{I}_{(m-n,n),(n-1,1)} &\tilde{I}_{(m-n,n),(n-2,2)}&\cdots & \tilde{I}_{(m-n,n),(n-n,n)}\\
\vdots\\
\tilde{I}_{(m-m,m),(n-0,0)} & \tilde{I}_{(m-m,m),(n-1,1)} &\tilde{I}_{(m-m,m),(n-2,2)}&\cdots & \tilde{I}_{(m-m,m),(n-n,n)}
\end{array}\right]
\end{eqnarray*}
with the dimensionality $(m+1) \times (n+1)$ for $m > n$.

\subsection{Powers of $\mathcal{X}$ and $\mathcal{Y}$}

The differentiation matrices $\mathcal{X}$ and $\mathcal{Y}$ are indexed over integers and they are not banded. Therefore, we cannot take for granted that the individual entries of their powers (necessary, for example, to approximate higher derivatives) are bounded. The boundedness of the (univariate) differentiation matrix $\mathcal D$ made up for a significant part of the narrative in \cite{iserles24ssm}. Given a univariate weight $w$, there exists  $s\in\mathbb{N}$ such that ${\mathcal D}^k$ is bounded for $k=0,\ldots,s$ and at least some entries of $\mathcal{D}^{s+1}$ are unbounded: it is then said that $\mathcal{D}$ has {\em index\/} $s$. The boundedness of the index is important because, once the support of $w$ is, for example, a bounded interval, an index equal to infinity would lead to exceedingly poor approximation \cite{iserles24ssm}.

The extension of this work to orthogonal systems in triangles will not be pursued in this paper. The method of proof in \cite{iserles24ssm} is unlikely to extend to bivariate setting and, at any rate, because $\tilde{I}^{\alpha,\beta,\gamma}_{(m-\ell,\ell),(n-k,k)}$ is at present unknown (although available in a recursive form),  the method of proof from  \cite{iserles24ssm} is probably unsuitable. Having said so, it is highly likely that an important consequence of  \cite{iserles24ssm} remains valid, namely that in the `ultraspherical case' $\alpha=\beta=\gamma$ the index increases as $\alpha$ grows. In Section~5 we argue that an optimal choice for analytic function $f$ is for the parameters $\alpha,\beta,\gamma$ to be even natural numbers and it is likely that, similarly to \cite{iserles24ssm}, the index of $\mathcal{X}$ and $\mathcal{Y}$ grows with $\alpha=\beta=\gamma$. Yet, this is matter for future research.

\subsection{Matrix-vector multiplication of the differentiation matrices}

Univariate W-systems allow for rapid numerical linear algebra: multiplication of a vector by the differentiation matrix (an action corresponding to differentiation), exponentiation of the differential matrix and solution of linear algebraic systems with a matrix polynomial in the differentiation matrix \cite{iserles24ssm}. The reason is that the differentiation matrix in two major instances (ultraspherical weights in $(-1,1)$, Laguerre weights in $(0,\infty)$) is either {\em semi-separable\/} of rank 1 or can be decomposed into `even' and `odd' matrices with this feature. Specifically, each submatrix of $\mathcal{D}$ strictly above (or, by skew symmetry, strictly below) the main diagonal is of  rank 1.

The situation is more complicated in our bivariate case. In this subsection, we describe fast matrix/vector multiplication for the differentiation matrices $\mathcal{X}$ and $\mathcal{Y}$. This is important because functions in the Hilbert space $\mathcal{H}$ are described in our setting by $\ell_2[\mathbb{Z}_+]$ vectors: once $f\in\mathcal{H}$ corresponds to $\MM{f}$, $\partial f/\partial x$ and $\partial f/\partial y$ correspond to the vectors $\mathcal{X}\MM{f}$ and $\mathcal{Y}\MM{f}$ respectively.

Bearing in mind (\ref{eq:3.6}), we separate $\mathcal{X}$ into two sub-matrices, denoted by $\mathcal{F}$ and $\mathcal{E}$, where
\begin{eqnarray*}
\mathcal{F}_{m, n}&=&\left[\alpha h^{\gamma, \beta}_k \delta_{\ell, k} I^{\alpha, \beta, \gamma}_{(m-k,k),(n-k,k)}\right]_{0 \leq k \leq \min \{m, n\}}\!,\qquad \mathcal{F} = \left[\mathcal{F}_{m, n}\right]_{0 \leq m, n\leq M}\!,\\
\mathcal{E}_{m, n}&=&\left[\gamma S^{\gamma, \beta}_{\ell, k} \tilde{I}^{\alpha, \beta, \gamma}_{(m-\ell,\ell),(n-k,k)}\right]_{0 \leq \ell \leq m, 0 \leq k \leq n}\!, \qquad \mathcal{E} = \left[\mathcal{E}_{m, n}\right]_{0 \leq m, n\leq M}\!.
\end{eqnarray*}
In particular, it holds that
\begin{eqnarray*}
\mathcal{F}_{(m-\ell, \ell),(m-k, k)} &=& \mathcal{E}_{(m-\ell, \ell),(m-k, k)} = 0, \quad \ell \neq k,\\
\mathcal{F}_{(m-k, k),(m-k, k)} &=& \mathcal{E}_{(m-k, k),(m-k, k)}, \quad \ell = k
\end{eqnarray*}
which also indicates the diagonal entries  of $\mathcal{X} = \mathcal{F} - \mathcal{E}$ vanish. Alternatively, $\mathcal{F}_{m, m} = \mathcal{E}_{m, m}$. Furthermore, since $\mathcal{X}$ is skew-symmetric, we set the truncated matrix with the symmetry by
\begin{eqnarray*}
\mathcal{F}_{M, M} &=& \left[\begin{array}{ccccc}
\mathcal{F}_{0, 0} & -\mathcal{F}_{1, 0}^\top & -\mathcal{F}_{2, 0}^\top &\cdots & -\mathcal{F}_{M, 0}^\top\\
\mathcal{F}_{1, 0} & \mathcal{F}_{1, 1} &-\mathcal{F}_{2, 1}^\top&\cdots & -\mathcal{F}_{M, 1}^\top\\
\mathcal{F}_{2, 0} & \mathcal{F}_{2, 1} &\mathcal{F}_{2, 2}&\cdots & -\mathcal{F}_{M, 2}^\top\\
\vdots\\
\mathcal{F}_{M, 0} & \mathcal{F}_{M, 1} &\mathcal{F}_{M, 2}&\cdots & \mathcal{F}_{M, M}
\end{array}\right]\!, \qquad \mathcal{F}_{n, m} = -\mathcal{F}_{m, n}^\top,
\end{eqnarray*}
where $\mathcal{F}_{m, n}$ is a diagonal matrix of $(m+1) \times (n+1)$ whose entries are
\begin{eqnarray*}
&& \mathcal{F}_{(m-k, k),(m-k, k)}=\alpha h^{\gamma, \beta}_k\mathfrak{a}_{m, k}\mathfrak{b}_{n, k}, \qquad k = 0, 1, \cdots, \min\{m, n\},\\
&& \mathfrak{a}_{m, k} = (-1)^m \frac{\alpha h^{\gamma, \beta}_k \mathrm{\Gamma}\!(\beta+\gamma+m+k+2)}
{\mathrm{\Gamma}\!(\alpha+\beta+\gamma+m+k+2)}, \quad  \mathfrak{b}_{n, k} = (-1)^n\frac{\mathrm{\Gamma}\!(\alpha+n-k+1)}{(n-k)!}
\end{eqnarray*}
from the expression (\ref{eq:3.9}). An illustrative structure of $\mathcal{F}$ is
\begin{eqnarray*}
\left[\begin{array}{c:cc:ccc:ccccccccccccccccccccccc}
\mathfrak{a}_{0, 0}\mathfrak{b}_{0, 0} &-\mathfrak{a}_{1, 0}\mathfrak{b}_{0, 0}& 0&-\mathfrak{a}_{2, 0}\mathfrak{b}_{0, 0} & 0& 0&\cdots\\
\hdashline
\mathfrak{a}_{1, 0}\mathfrak{b}_{0, 0} & \mathfrak{a}_{1, 0}\mathfrak{b}_{1, 0} &0 & \mathfrak{a}_{2, 0}\mathfrak{b}_{1, 0} & 0 & 0&\cdots\\
0 & 0 &\mathfrak{a}_{1, 1}\mathfrak{b}_{1, 1} &0 & -\mathfrak{a}_{1, 1}\mathfrak{b}_{0, 1}& 0&\cdots\\
\hdashline
\mathfrak{a}_{2, 0}\mathfrak{b}_{0, 0} & \mathfrak{a}_{2, 0}\mathfrak{b}_{1, 0} &0&\mathfrak{a}_{2, 0}\mathfrak{b}_{2, 0}&0&0\\
0 & 0 &\mathfrak{a}_{2, 1}\mathfrak{b}_{0, 1}&0&\mathfrak{a}_{2, 1}\mathfrak{b}_{1, 1}&0\\
0 &0 & 0&0&0&\mathfrak{a}_{2, 2}\mathfrak{b}_{2, 2}&\cdots \\
\hdashline
\mathfrak{a}_{3, 0}\mathfrak{b}_{0, 0} & \mathfrak{a}_{3, 0}\mathfrak{b}_{1, 0} &0&  \mathfrak{a}_{3, 0}\mathfrak{b}_{2, 0} &0&0&&&&\\
0 & 0 &\mathfrak{a}_{3, 1}\mathfrak{b}_{1, 1}& 0 &\mathfrak{a}_{3, 1}\mathfrak{b}_{2, 1}&0&&&&\\
0 & 0 &0& 0 &0&\mathfrak{a}_{3, 2}\mathfrak{b}_{2, 2}&&&&\\
0 & 0 &0& 0 &0&0&\cdots&&&\\
\hdashline
\vdots&&\vdots&&&\vdots&\vdots&&&&&&
\end{array}\right]\!.
\end{eqnarray*}
We let
\begin{eqnarray*}
\MM{g} &=& \mathcal{F}_{M, M} \MM{f}, \qquad \MM{g}_M = [\MM{g}_0,\MM{g}_1, \cdots, \MM{g}_M]^\top,\qquad
\MM{f}_M = [\MM{f}_0,\MM{f}_1, \cdots, \MM{f}_M]^\top,\\
\MM{g}_m &=& [g_{m, 0}, g_{m, 1}, \cdots, g_{m, m}]^\top, \qquad \MM{f}_n = [f_{n, 0}, f_{n, 1}, \cdots, f_{n, n}]^\top.
\end{eqnarray*}
Assume that $\mathfrak{a}_{m, k}$, $\mathfrak{b}_{n,k}$ and the initial scalar value
\begin{displaymath}
\MM{g}_0 = g_{0, 0} = -\mathfrak{b}_{0, 0} \sum_{m=1}^N \mathfrak{a}_{m, 0}f_{m,0}
\end{displaymath}
are computed in advance. For each $\MM{g}_m$, the diagonally sparse structure motivates the recursive form
\begin{eqnarray*}
g_{m, j} &=& \mathfrak{a}_{m, j} \sum_{n=j}^m \mathfrak{b}_{n, j}f_{n, j} - \mathfrak{b}_{m, j} \sum_{n=m+1}^N \mathfrak{a}_{n, j} f_{n, j}\\
&:=& \mathfrak{a}_{m, j} \sigma_{m, j} - \mathfrak{b}_{m, j} \rho_{m, j},\qquad j = 0, 1, \cdots, m,\\
\mbox{where}\quad \sigma_{m, j} &=& \sum_{n=j}^m \mathfrak{b}_{n, j}f_{n, j}, \quad \rho_{m, j} = \sum_{n=m+1}^N \mathfrak{a}_{n, j} f_{n, j}.
\end{eqnarray*}
The recursive relations with the initial values are exhibited as
\begin{eqnarray*}
\sigma_{m, j} &=& \sigma_{m-1, j} + \mathfrak{b}_{m, j} f_{m, j}, \qquad \rho_{m, j} = \rho_{m-1, j} - \mathfrak{a}_{m, j} f_{m, j}, \qquad m \geq 1,\\
\sigma_{0, j}&=& \mathfrak{b}_{0, 0}f_{0, 0}, \qquad \rho_{0, 0} = \sum_{n=1}^N \mathfrak{a}_{n, 0} f_{n, 0}
\end{eqnarray*}
Counting the  cost of each $\MM{g}_m$ results in $4(m+1)$ flops which implies that the entire cost of a matrix/vector multiplication is $\sum_{m=0}^{M}4(m+1) = 2(M+1)(M+2)$. This should be compared with the dimension $\frac12 (M+1)(M+2) \times\frac12 (M+1)(M+2)$ of $\mathcal{F}_{M, M}$.

The other sub-matrix is $\mathcal{E}$ composed of $\gamma S^{\gamma, \beta}_{\ell, k} \tilde{I}^{\alpha, \beta, \gamma}_{(m-\ell,\ell),(n-k,k)}$. Furthermore, the property
\begin{displaymath}
\tilde{I}^{\alpha, \beta, \gamma}_{(m-\ell,\ell),(m-k,k)} = 0, \qquad \ell \neq k
\end{displaymath}
implies that $\mathcal{E}_{m, m}$ is diagonal. Although $S^{\gamma, \beta}_{\ell, k}$ has different expressions for $\ell \geq k$ and $\ell < l$, we only take the form (\ref{eq:3.7}) because of $\tilde{I}^{\alpha, \beta, \gamma}_{(m-\ell,\ell),(n-k,k)} = 0$ for $\ell < k$. We write the illustrative structure of $\mathcal{E}$ truncated by $M$,
\begin{eqnarray*}
\mathcal{E}_{M, M} &=& \left[\begin{array}{ccccc}
\mathcal{E}_{0, 0} & -\mathcal{E}_{1, 0}^\top &-\mathcal{E}_{2, 0}^\top &\cdots & -\mathcal{E}_{M, 0}^\top\\
\mathcal{E}_{1, 0} & \mathcal{E}_{1, 1} &-\mathcal{E}_{2, 1}^\top &\cdots & -\mathcal{E}_{M, 1}^\top\\
\mathcal{E}_{2, 0} & \mathcal{E}_{2, 1} &\mathcal{E}_{2, 2}&\cdots & -\mathcal{E}_{M, 2}^\top\\
\vdots&\vdots & \vdots & & \vdots\\
\mathcal{E}_{M, 0} & \mathcal{E}_{M, 1} &\mathcal{E}_{M, 2}&\cdots & \mathcal{E}_{M, M}\\
\end{array}\right]\!,
\end{eqnarray*}
where $\mathcal{E}_{m, n}$ is a square matrix with dimension $(m+1) \times (n+1)$ and proceed with the matrix vector multiplication, letting first
\begin{displaymath}
\MM{h} = \mathcal{E}_{M, M} \MM{f}, \qquad \MM{h}_M = [\MM{h}_0,\MM{h}_1, \cdots, \MM{h}_M]^\top\!,\qquad
\MM{h}_m = [h_{m, 0}, h_{m, 1}, \cdots, h_{m, m}]^\top\!.\end{displaymath}
From the recursive relation
\begin{eqnarray*}
\mathcal{E}_{(m-\ell, \ell),(n-k, k)} &=& q_{m,\ell} \mathcal{E}_{((m-1)-\ell,\ell),(n-k,k)} + \tilde{d}^{[1]}_{m, \ell} \mathcal{E}_{((m-1)-(\ell-1),\ell-1),(n-k,k)}\\
&&\mbox{} - \tilde{d}^{[2]}_{m, \ell} \mathcal{E}_{(m-(\ell-1),\ell-1),(n-k,k)},\\
\tilde{d}^{[1]}_{m, \ell} &=& \frac{\beta+\ell}{\beta+\gamma+\ell}d^{[1]}_{m, \ell},\qquad \tilde{d}^{[2]}_{m, \ell} = \frac{\beta+\ell}{\beta+\gamma+\ell} d^{[2]}_{m, \ell}.
\end{eqnarray*}
we rewrite the recursive formula from Theorem 2 in a matrix form,
\begin{eqnarray*}
\mathcal{E}_{m, n} &=& \left[\begin{array}{ccccc}
0 & 0 & 0 & \cdots & 0\\
-\tilde{d}^{[2]}_{m, 1} &0& 0 & \cdots& 0 \\
0 & -\tilde{d}^{[2]}_{m, 2} & \ddots & \ddots & 0 \\
\vdots & \ddots & \ddots & & \vdots\\
0 & \cdots & 0 & -\tilde{d}^{[2]}_{m, m}& 0 \\
\end{array}\right]\!\mathcal{E}_{m, n} \\
&&\mbox{}+
\left[\begin{array}{ccccc}
q_{m,0} & 0 & 0& \cdots & 0\\
\tilde{d}^{[1]}_{m, 1} &q_{m,1}& 0 &\cdots& 0 \\
0 & \tilde{d}^{[1]}_{m, 2} & \ddots &  & 0 \\
\vdots & \ddots & \ddots & & \vdots\\
0 & \cdots & 0 & \tilde{d}^{[1]}_{m, m}& q_{m,m} \\
\end{array}\right]\left[\begin{array}{c}
\mathcal{E}_{m-1, n}\\
0
\end{array}\right]\!.
\end{eqnarray*}
Therefore
\begin{displaymath}
\left[\begin{array}{llllllll}
1 & 0 & \cdots & 0\\
\tilde{d}^{[2]}_{m, 1} &1& \ddots& 0 \\
0 & \tilde{d}^{[2]}_{m, 2} & \ddots & 0 \\
\vdots & \ddots & \ddots\\
0 & \cdots & \tilde{d}^{[2]}_{m, m}& 1 \\
\end{array}\right]\mathcal{E}_{m, n}\MM{f}_n =
\left[\begin{array}{llllllll}
q_{m,0} & 0 & \cdots & 0\\
\tilde{d}^{[1]}_{m, 1} &q_{m,1}& \ddots& 0 \\
0 & \tilde{d}^{[1]}_{m, 2} & \ddots & 0 \\
\vdots & \ddots & \ddots\\
0 & \cdots & \tilde{d}^{[1]}_{m, m}& q_{m,m} \\
\end{array}\right]\!\left[\begin{array}{c}
\mathcal{E}_{m-1, n}\MM{f}_n\\
0
\end{array}\right]\!,
\end{displaymath}
which we rewrite as
\begin{displaymath}
V_m \mathcal{E}_{m, n}\MM{f}_n = U_m \left[\begin{array}{c}
\mathcal{E}_{m-1, n}\MM{f}_n\\
0
\end{array}\right]\!, \quad m = 0, 1, \cdots, M,
\end{displaymath}
where both $V_m$ and $U_m$ are lower bidiagonal matrices
\begin{eqnarray*}
V_m = \left[\begin{array}{cccc}
1 & 0 & \cdots & 0\\
\tilde{d}^{[2]}_{m, 1} &1& \ddots& 0 \\
0 & \tilde{d}^{[2]}_{m, 2} & \ddots & 0 \\
\vdots & \ddots & \ddots\\
0 & \cdots & \tilde{d}^{[2]}_{m, m}& 1 \\
\end{array}\right]\!,\qquad
U_m = \left[\begin{array}{cccc}
q_{m,0} & 0 & \cdots & 0\\
\tilde{d}^{[1]}_{m, 1} &q_{m,1}& \ddots& 0 \\
0 & \tilde{d}^{[1]}_{m, 2} & \ddots & 0 \\
\vdots & \ddots & \ddots\\
0 & \cdots & \tilde{d}^{[1]}_{m, m}& q_{m,m} \\
\end{array}\right]\!
\end{eqnarray*}
with the initial value
\begin{displaymath}
\mathcal{E}_{0, 0} = \gamma S^{\gamma, \beta}_{0, 0} \tilde{I}^{\alpha,\beta, \gamma}_{(0, 0),(0, 0)} = \gamma \frac{\mathrm{\Gamma}\!\left(\alpha+1\right)
\mathrm{\Gamma}\!\left(\beta+\gamma+1\right)}{\mathrm{\Gamma}\!\left(\alpha+\beta+\gamma+2\right)}
\frac{\mathrm{\Gamma}\!\left(\beta+1\right)
\mathrm{\Gamma}\!\left(\gamma\right)}{\mathrm{\Gamma}\!\left(\beta+\gamma+1\right)}.
\end{displaymath}

To figure out the relation between $\mathcal{E}_{n, m}^\top$ and $\mathcal{E}_{n, m-1}^\top$ in the upper triangular part of $\mathcal{E}_{M, M}$, the second recursive relation in Lemma \ref{lemma:2} is applied to get
\begin{eqnarray*}
V_m \mathcal{E}_{n, m}^\top \MM{f}_n = U_m \left[\begin{array}{c}
\mathcal{E}_{n, m-1}^\top \MM{f}_n\\
0
\end{array}\right]\!, \quad m = 1, \cdots, M.
\end{eqnarray*}

We apply the two recursive relations together to the matrix vector $\mathcal{E}_{M, M}\MM{f}_M$ to the lower triangular and upper triangular parts, respectively. This results in
\begin{eqnarray}
\label{eq:3.11}
V_m \left[\sum_{n=0}^{m-1} \mathcal{E}_{m, n} \MM{f}_n\right] &=&
U_m \left[\begin{array}{llllllll}
\sum_{n=0}^{m-1} \mathcal{E}_{m-1, n} \MM{f}_n\\
0
\end{array}\right]\!,\nonumber\\
V_m \left[\sum_{n=m+1}^M \mathcal{E}_{n, m}^\top \MM{f}_n\right] &=&
U_m \left[\begin{array}{llllllll}
\sum_{n=m+1}^M \mathcal{E}_{n, m-1}^\top \MM{f}_n\\
0
\end{array}\right]
\end{eqnarray}
bearing a computational cost of $3m+1$ operations.

Each $\MM{h}_m$ is represented in the form
\begin{eqnarray*}
\MM{h}_m = \sum_{n=0}^{m-1} \mathcal{E}_{m, n} \MM{f}_n + \mathcal{E}_{m, m} \MM{f}_m - \sum_{n=m+1}^M \mathcal{E}_{n, m}^\top \MM{f}_n\\
\end{eqnarray*}
The two block vectors
\begin{displaymath}
\mathcal{E}_{0, 0} \MM{f}_{0, 0}, \qquad - \sum_{n=1}^M \mathcal{E}_{n, 0}^\top \MM{f}_n
\end{displaymath}
are calculated in advance. For each $m$, in order to get the two sum terms in $\MM{h}_m$ from the previous terms,
\begin{eqnarray*}
\left[\begin{array}{c}
\sum_{n=0}^{m-1} \mathcal{E}_{m-1, n} \MM{f}_n\\
0
\end{array}\right]\! \qquad\mbox{and}\qquad \left[\begin{array}{c}
\sum_{n=m+1}^M \mathcal{E}_{n, m-1}^\top \MM{f}_n\\
0
\end{array}\right]\!,
\end{eqnarray*}
 we need  to perform the operation (\ref{eq:3.11}) just once at the cost of $3m+1$ operations. The remainder,  $\mathcal{E}_{m, m} \MM{f}_m$, needs further $m+1$ operations since the matrix $\mathcal{E}_{m, m}$ is diagonal. Finally, the whole cost is
\begin{displaymath}
\sum_{m=1}^M (4m+2) = 2M(M+2).
\end{displaymath}
In contrast to the dimensionality of the truncated matrix $\mathcal{E}_{M, N}$, the cost is optimal, being proportional to the number of rows.

The above construction applies for the differentiation matrix $\mathcal{X}$. but an identical procedure applies to the matrix  $\mathcal{Y}$. Specifically, the expression (\ref{eq:3.6}) is similar to $\mathcal{E}$ and the diagonal block matrix is a zero matrix which means that the cost is $M\left(3M+5\right)/2$, again using the recursive algorithm.%%%%%%%%%

\section{Convergence of Koornwinder-type W-functions on a triangle}

The immediate purpose of constructing the W-system $\{\varphi_{n, k} (x_1, x_2)\}_{n \in\mathbb{Z}_+,\; k=0,\ldots,n}$ is to approximate  bivariate $\mathrm{L}_2(\mathcal{T})$  functions $f(x_1, x_2)$ which obey zero boundary conditions along $\partial \mathcal T$. Within the context of this paper, we intend this approximation to be used as a major ingredient of a spectral method, while acknowledging its wider applications. The main weapon in our endeavour, to construct fast-approximating W-systems, is the choice of the parameters $\alpha,\beta,\gamma>0$.

In a univariate setting this has been already considered for $\alpha=\beta>0$ in \cite{iserles24ssm} and our analysis threads the same ground. Thus, we associate to $f\in\mathrm{L}_2(\mathcal{T})$ the sequence
\begin{displaymath}
  \hat{\MM{f}}=\{\hat{f}_{m,\ell}\}_{m\in\mathbb{Z}_+,\; \ell=0,\ldots,m},\qquad\mbox{where}\qquad \hat{f}_{m,\ell}=\int_{\mathcal T} f(x,y)\varphi_{m,\ell}(x,y)\D x\D y.
\end{displaymath}
We are concerned with the convergence of
\begin{displaymath}
  g_M:=\sum_{m=0}^M \sum_{\ell=0}^m \hat{f}_{m,\ell}\varphi_{m,\ell}
\end{displaymath}
to $f$ (we need to assume  $f\in\mathrm{L}_2$ but for present purposes it is easier to stipulate that $f$ is analytic in the closure of $\mathcal{T}$ or, at the very least, lives in $\mathrm{C}^p(\mathcal{T})$ for sufficiently large $p$) as $M\rightarrow\infty$. Note however that
\begin{displaymath}
  \hat{f}_{m,\ell} =\int_{\mathcal{T}} f(x,y) \sqrt{w(x,y)}p_{m,\ell}(x,y)\D x\D y=\int_{\mathcal{T}} \frac{f(x,y)}{\sqrt{w(x,y)}} p_m(x,y) w(x,y)\D x\D y.
\end{displaymath}
Therefore, the speed of convergence for $f$ in the context of W-systems is identical to that of $f/\sqrt{w}$ in the context of orthogonal polynomials. Note that $w>0$ in the interior of $\mathcal T$, while both $f=0$ and $w=0$ along its boundary.

Once $f$ is an analytic function, the convergence of $g_M$ to $f$ is very fast: in the univariate case, in the $\mathrm{L}_2$ norm, it is {\em spectral:\/} $\limsup_{M\rightarrow\infty} \|f-g_M\|_{\mathrm{L}_2}^{1/M}\in[0,1)$ and, although we are not aware of general results of this kind in a triangle, there are good reasons to suspect that this remains the case. Thus, we should choose $\alpha,\beta,\gamma>0$ so that $f/\sqrt{w}$ remains analytic. A good choice is $\alpha=\beta=\gamma=2$ because then also $\sqrt{w}$ is analytic, while the singularity along the boundary caused by the division by $\sqrt{w}$ is removable.

With greater generality, let $f(x,y)=x^\kappa y^\theta (1-x-y)^\eta \tilde{f}(x,y)$, where $\kappa,\theta,\eta>0$ are non-integer, while $\tilde{f}$ is analytic in $\mathcal{T}$. Then the choice $\alpha=2(1+\kappa)$, $\beta=2(1+\theta)$ and $\gamma=2(1+\eta)$ means that $f/\sqrt{w}$ is analytic!

As an example of a function with weak singularity along $\partial\mathcal{T}$ we let
\begin{eqnarray*}
f(x, y) &=& \ee^{x-2y} \sqrt{x y(1-x-y)}, \quad (x, y) \in \mathcal{T},\\
f(x, y) &=& 0, \quad \text{on} \quad \partial \mathcal{T}, \\
\mathcal{T}&=& \{0 \leq x \leq 1, \ 0 \leq y \leq 1-x\}.
\end{eqnarray*}
The approximate series is
\begin{eqnarray*}
f(x, y) &\approx& \sum_{n=0}^{\mathrm{n_{\max}}} \sum_{k=0}^n f_{n, k} \varphi_{n, k}(x, y),\\
f_{n, k} &=& \int_\mathcal{T} f(x, y) \varphi_{n, k}(x, y)\D x \D y\nonumber\\
&=&\int_\mathcal{T} f(x, y) \sqrt{\omega(x, y)} p_{n, k}(x, y)\D x \D y,\nonumber
\end{eqnarray*}
where $\omega(x, y)$ and $p_{n, k}(x, y)$ feature at the beginning of Subsection 3.2. To illustrate the error, we define the discrete infinite and $\ell_2$ errors as
\begin{eqnarray*}
\MM{e}_\infty &:=& \max \left|e(x_i, y_j)\right|, \qquad \MM{e}_2 := \sqrt{\sum_{i=0}^M \sum_{j=0}^{M-i} \left(e(x_i, y_j)\right)^2},\\
x_i &=& \frac12\!\left(1-\cos \frac{i \pi}{M}\right)\!, \qquad y_j = \frac12\!\left(1-\cos \frac{j \pi}{M}\right)\!, \qquad M=4, \qquad \mathrm{n_{\max}}=8.
\end{eqnarray*}

For a fixed $\mathrm{n_{\max}}$, the number of expansion terms is $N=\sum_{k=0}^{\mathrm{n_{\max}}}(k+1) = (\mathrm{n_{\max}}+1)(\mathrm{n_{\max}}+2)/2$. We choose the parameters $\alpha=\beta=\gamma=1$. In Fig.~\ref{Fig:4.1}, the magnitude of $f_{n, k}$ is displayed in the left figure. We rearrange the coefficients into a single sequence with a single subscript $N$, i.e.\ denoting the coefficients by $f_N$. The cut-off parameter $\mathrm{n_{\max}}=8$ determines the maximum $N$. It can be observed that the coefficients decay rapidly, at a spectral speed. For instance, for $N=45$ the coefficient is of an order of magnitude of $10^{-9}$, in spite of the weak singularity on the boundary.%%%%

 \begin{figure}[tb]
  \begin{center}
    \includegraphics[width=120pt]{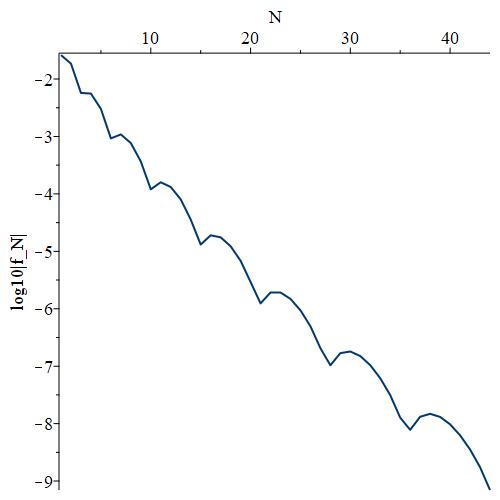}\hspace*{5pt}
    \includegraphics[width=120pt]{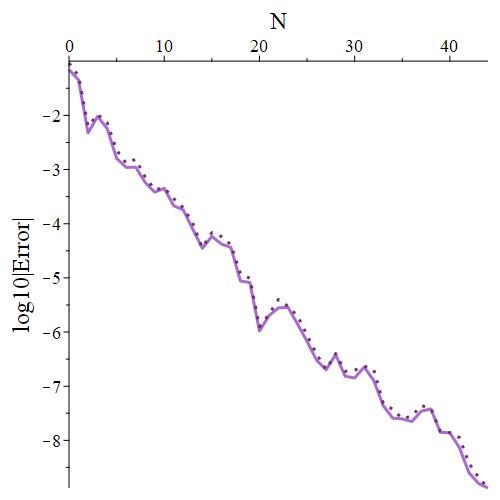}
    \caption{The coefficients $f_{n, k}$, drawn to logarithmic scale (left) and the $\ell_\infty$ (solid) and $\ell_2$ (dotted) errors (right) for $f(x, y) = \ee^{x_1-2x_2} \sqrt{x_1x_2(1-x_1-x_2)}$, $\alpha=\beta=\gamma=1$, $N=1, \cdots, 45$.}
\label{Fig:4.1}
  \end{center}
\end{figure}

The $l_\infty$ (solid line) and $l_2$ (dotted line) error norms as $N$ increases are depicted on the right side of Fig.~\ref{Fig:4.1} . It can be seen that the error $e_2$ becomes as small as $10^{-9}$ for $N=45$. %%%%%%

The point errors are plotted in Fig.~\ref{Fig:4.2} for $N=5, 20, 35, 45$, corresponding to the top left, top right, bottom left and bottom right. It is observed that the absolute error consistently and rapidly decreases with increasing  $N$, consistently with the theoretical analysis.

 \begin{figure}[tb]
  \begin{center}
    \hspace*{-10pt}\includegraphics[width=120pt]{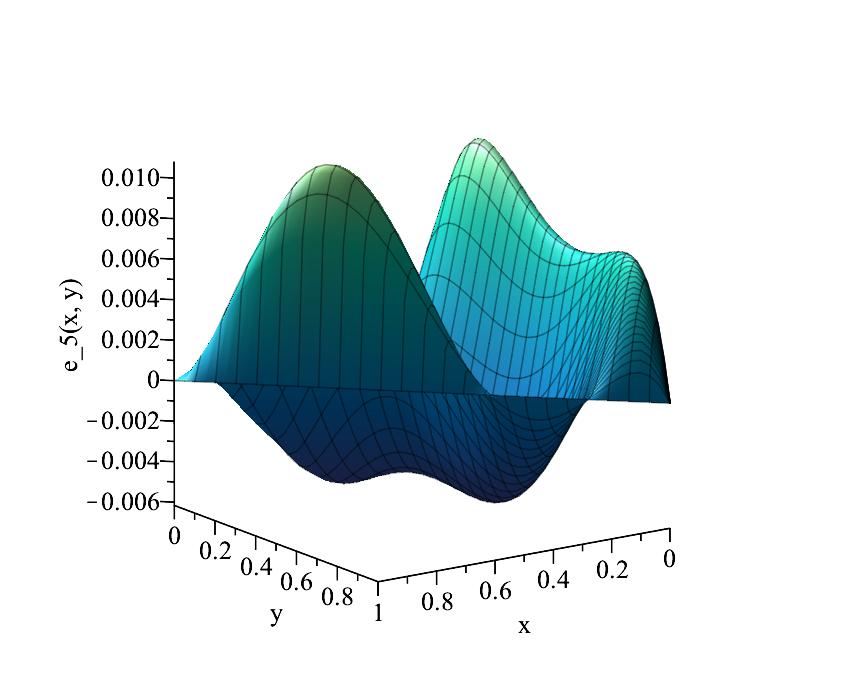}\hspace*{5pt}
    \includegraphics[width=120pt]{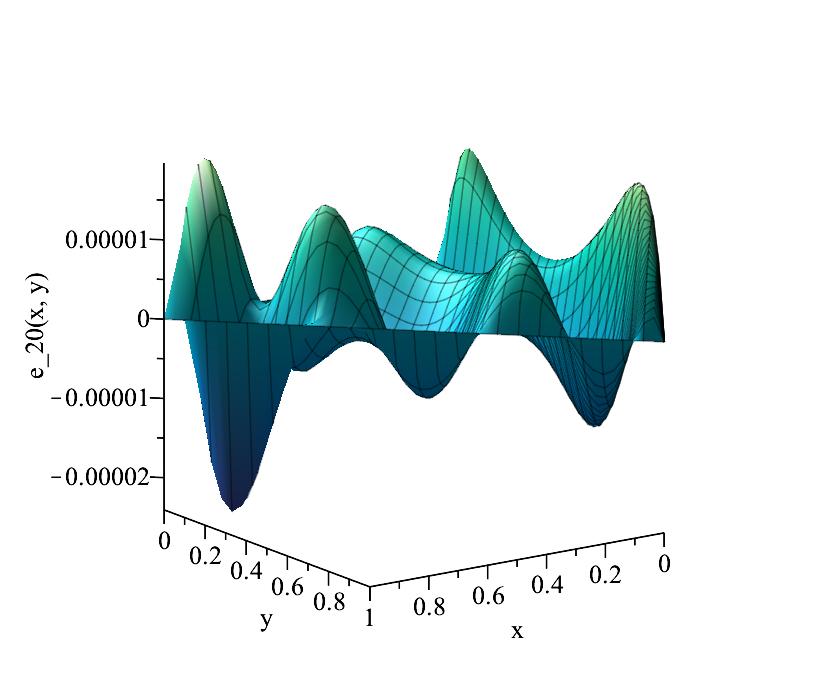}\\
    \hspace*{-10pt}\includegraphics[width=120pt]{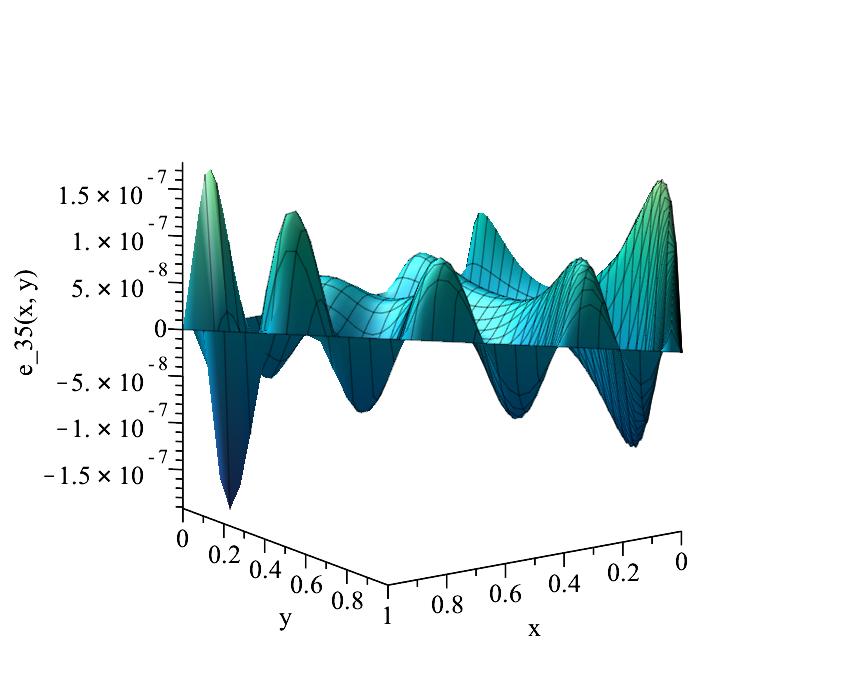}\hspace*{5pt}
    \includegraphics[width=120pt]{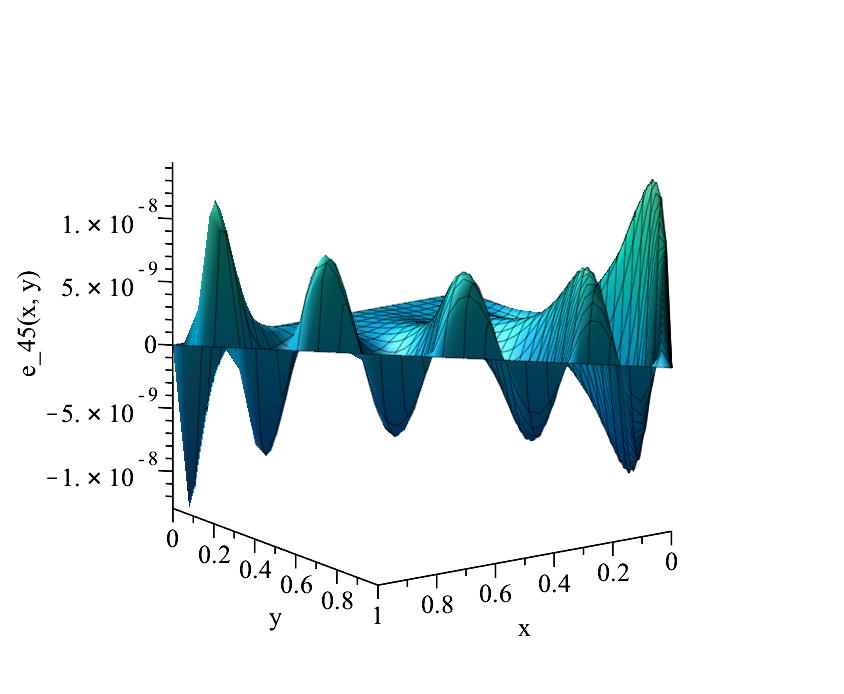}\\
    \caption{Pointwise errors $e_N(x, y)$ for approximating $f(x, y) = \ee^{x_1-2x_2} \sqrt{x_1x_2(1-x_1-x_2)}$, $\alpha=\beta=\gamma=1$ with $N=5$ (top left), $N=20$ (top right),$N=35$ (bottom left), $N=45$ (bottom right) respectively.}
    \label{Fig:4.2}
  \end{center}
\end{figure}

To understand better how to choose the parameters let us consider another example.
We maintain all conditions from Example~1 except for $\alpha=\beta=\gamma=2$ and $n_{\max}=6$. In this case, the basis functions are analytic, unlike the  function $f$ which is being approximated. The coefficients $f_N$ are displayed in the left (scaled by $N$) and middle (scaled by $N^2$) in Fig. \ref{Fig:4.3}. The illustrative curves show that the decay rate of the coefficients is approximately $\mathcal{O}(N)$, rather than at an exponential speed. The right figure exhibits the errors $e_2$ (dotted line) and $e_{\infty}$ (solid line) which do not converge well. The conclusion is that it is the interplay between the regularity of $f$ and of the $\varphi_{n,k}$s, as described earlier, that makes the difference: $f\varphi_{n,k}$ must be analytic in the closure of $\mathcal{T}$ to ensure spectral convergence.

\begin{figure}[tb]
  \begin{center}
    \includegraphics[width=120pt]{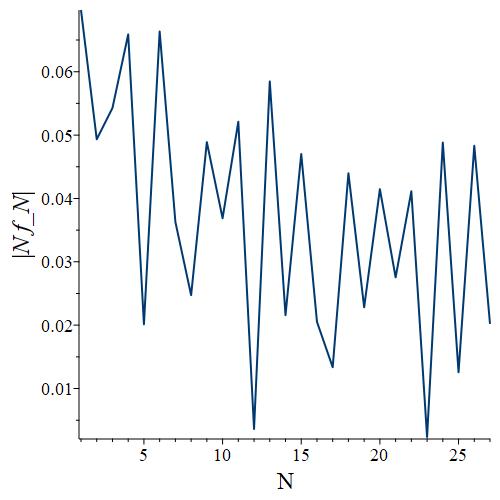}\hspace*{2pt}
    \includegraphics[width=120pt]{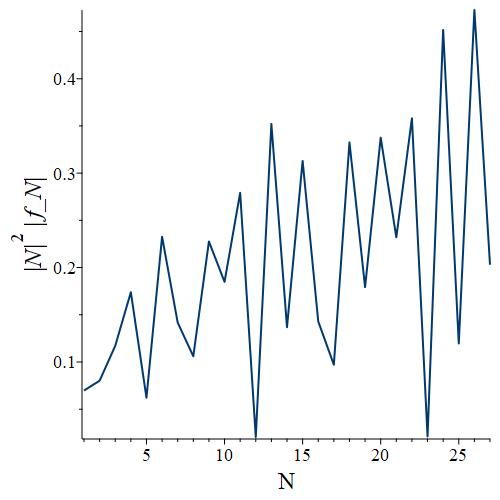}\hspace*{2pt}
    \includegraphics[width=120pt]{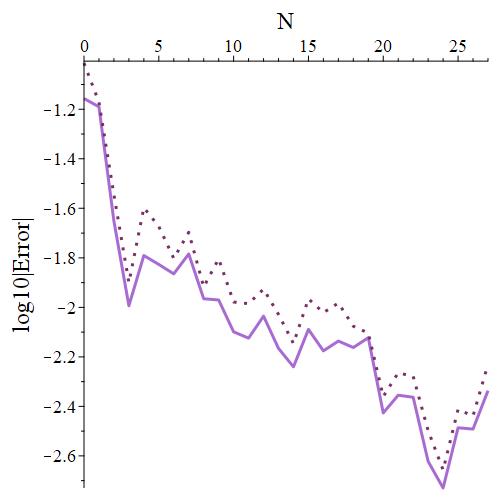}
    \caption{The coefficients $f_{n, k}$ scaled by $N$ (the left) and by $N^2$ (the middle) are plotted. The errors $e_\infty$ (solid) and $e_2$ (dotted) occur in the right figure for $f(x, y) = \ee^{x_1-2x_2} \sqrt{x_1x_2(1-x_1-x_2)}$, $\alpha=\beta=\gamma=2$ and $N=1, \cdots, 45$.}
\label{Fig:4.3}
  \end{center}
\end{figure}

To compare the convergence  of the W-system $\varphi_{n, k}$ and  orthogonal polynomials $p_{n, k}$ on the triangle $\mathcal{T}$ we introduce an additional test function,
\begin{displaymath}
f(x,y)=x(1-\ee^y)\sin(\pi(1-x-y)), \qquad \alpha=\beta=\gamma=2.
\end{displaymath}
Other parameters are chosen as the same in Fig.~\ref{Fig:4.1}. The coefficients $f_N$ and the errors $e_2$ and $e_\infty$ are shown in Fig.~\ref{Fig:4.4}: W-functions in the top row and orthogonal polynomials in the bottom one. We observe that $f_N$, $e_2$ and $e_{\infty}$ decrease with increasing $N$ at an  exponential speed. However, the W-system exhibits superior convergence to orthogonal polynomials on the triangle.

\begin{figure}[tb]
  \begin{center}
    \includegraphics[width=120pt]{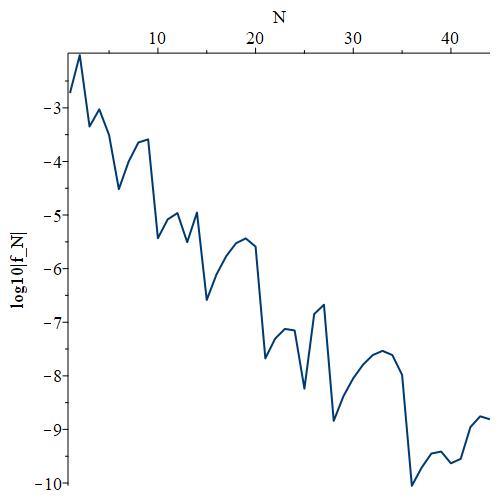}\hspace*{5pt}
    \includegraphics[width=120pt]{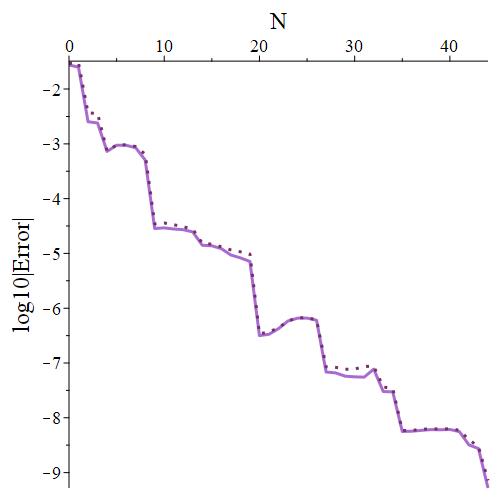}\\
     \includegraphics[width=120pt]{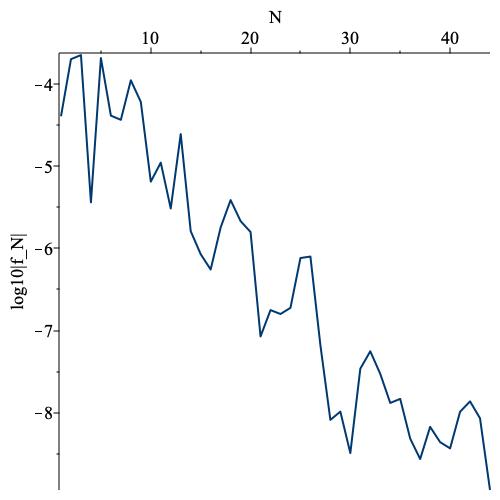}\hspace*{5pt}
    \includegraphics[width=120pt]{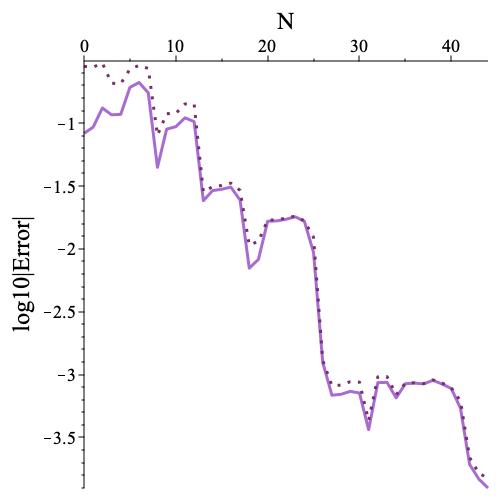}\\
    \caption{The coefficients $f_{n, k}$ are shown in the left column, and the errors $\ell_\infty$ (solid), $\ell_2$ (dotted) errors (the right) are in the right column, for $f(x, y) = x(1-\ee^y)\sin(\pi(1-x-y))$, $\alpha=\beta=\gamma=2$ by the W-system (the top row) and the orthogonal polynomial (the bottom row).}
\label{Fig:4.4}
  \end{center}
\end{figure}

\section{General Dirichlet boundary conditions}

A major assumption in the definition of W-systems is that they (and, by implication, the differential equations that we endevour to discretise) obey zero Dirichlet boundary conditions. In one dimension this assumption is not very restrictive. For example, given a PDE
\begin{displaymath}
  \frac{\partial u}{\partial t}={\mathcal L}u+f(t,u),\qquad t\geq0,\quad x\in[-1,1],
\end{displaymath}
where $\mathcal{L}$ is a linear differential operator, given in tandem with an initial condition $u(x,0)=u_0(x)$ and the boundary conditions $u(-1,t)=a_-(t)$, $u(1,t)=a_+(t)$, $t\geq0$, we construct for every $t\geq0$ a linear interpolation $\mu(x,t)=\frac12(1-x)a_-(t)+\frac12(1+x)a_+(t)$ and let $v(x,t)=u(x,t)-\mu(x,t)$.  Then
\begin{equation}
  \label{eq:6.1}
  \frac{\partial v}{\partial t}={\mathcal L}v+f(t,v+\mu)+\mathcal{L}\mu-\frac{\partial\mu}{\partial t}
\end{equation}
with the initial condition $v(x,0)=u_0(x)-\mu(x,0)$ and {\em zero\/} Dirichlet boundary conditions. All that is required is to solve (\ref{eq:6.1}) and let $u(x,t)=v(x,t)+\mu(x,t)$.

Similar idea works also in a multivariate setting, except that the formation of the `boundary function' $\mu$ is considerably more complicated, even in a right triangle. In this paper we adopt similar reasoning to \cite{huybrechs11fh5}, based upon ideas originating in computer-aided geometric modelling. First, however, we note that the benefits of this approach are not restricted to spectral methods for differential equations given in a triangle with non-zero Dirichlet boundary conditions but also extend to spectral element methods \cite{karniadakis08sea} and to $p$ and $hp$-adapted finite elements \cite{babuska88php}.

Let $(x,y)\in\mathcal{T}$, assume that the function $\mu$ is specified on $\partial\mathcal T$ and consider the sketch
\begin{displaymath}
  \begin{picture}(180,200)(0,20)
    \thicklines
    \put (20,20) {\line(1,0){180}}
    \put (20,20) {\line(0,1){180}}
    \put (200,20) {\line(-1,1){180}}
    \thinlines
    \put (50,120) {\circle*{4}}
    \put (50,20) {\line(0,1){150}}
    \put (20,120) {\line(1,0){80}}
    \put (20,150) {\line(1,-1){130}}
    \put (53,125) {\small$(x,y)$}
    \put (3,148) {\small $p_{B_3}$}
    \put (3,118) {\small $p_{B_1}$}
    \put (43,10) {\small $p_{A_2}$}
    \put (146,10) {\small $p_{A_3}$}
    \put (43,180) {\small $p_{C_2}$}
    \put (105,118) {\small $p_{C_1}$}
  \end{picture}
\end{displaymath}
We draw three straight lines across point $(x,y)\in\mathcal{T}$, in parallel with the faces of the triangle, terminating at the points
\begin{eqnarray*}
  p_{A_2}(x,y)&=&(x,0),\qquad p_{A_3}(x,y)=(x+y,0),\\
  p_{B_1}(x,y)&=&(0,y),\qquad p_{B_3}(x,y)=(0,x+y),\\
  p_{C_1}(x,y)&=&(1-y,y),\qquad p_{C_2}(x,y)=(x,1-x).
\end{eqnarray*}
We denote by $q_A$ the linear interpolation to the boundary function $\mu$ between $B_1$ and $C_1$, by $q_B$ the linear interpolation between $A_2$ and $C_2$ and by $q_C$ the linear interpolation between $A_3$ and $B_3$. Thus,
\begin{eqnarray*}
  q_A(x,y)&=&\frac{1-x-y}{1-y} \mu(0,y)+\frac{x}{1-y} \mu(1-y,y),\\
  q_B(x,y)&=&\frac{1-x-y}{1-x} \mu(x,0)+\frac{y}{1-x} \mu(x,1-x),\\
  q_C(x,y)&=&\frac{x}{x+y} \mu(x+y,0)+\frac{y}{x+y}\mu(0,x+y).
\end{eqnarray*}
Note that $q_A$ coincides with $\mu$ along two edges of $\mathcal{T}$ (for $x=0$ and for $x+y=1$) and is a linear function along the third edge, $y=0$. The same is true (with different edges) for $q_B$ and $q_C$. It follows that $\frac12 (q_A+q_B+q_C)$ coincides with $\mu$ on the boundary, except for a linear function. This linear function,
\begin{displaymath}
  \frac{\mu(0,0)}{2}+\frac{\mu(1,0)-\mu(0,0)}{2}x+\frac{\mu(0,1)-\mu(0,0)}{2}y,
\end{displaymath}
is extracted. 

To sum up, for every $(x,y)\in\mathcal{T}$ the interpolation is
\begin{eqnarray*}
  \mu(x,y)&=&\frac{1-x-y}{2(1-x)}\mu(x,0)+\frac{1-x-y}{2(1-y)}\mu(0,y)+\frac{y}{2(1-x)} \mu(x,1-x)\\
  &&\mbox{}+\frac{x}{2(1-y)}\mu(1-y,y)+\frac{x}{2(x+y)}\mu(x+y,0)+\frac{y}{2(x+y)}\mu(0,x+y)\\
  &&-\frac{1-x-y}{2}\mu(0,0)-\frac{x}{2}\mu(1,0)-\frac{y}{2}\mu(0,1)\\
  &=&\frac{1-x-y}{2}\!\left[\frac{\mu(x,0)}{1-x}+\frac{\mu(0,y)}{1-y}-\mu(0,0)\right]\\
  &&\mbox{}+\frac{x}{2}\!\left[\frac{\mu(1-y,y)}{1-y}+\frac{\mu(x+y,0)}{x+y}-\mu(1,0)\right]\\
  &&\mbox{}+\frac{y}{2}\!\left[\frac{\mu(x,1-x)}{1-x}+\frac{\mu(0,x+y)}{x+y}-\mu(0,1)\right]
\end{eqnarray*}
and it coincides with $\mu$ along $\partial T$. Note that, given $(x,y)\in\mathcal T$, the function $\mu$ uses nine values on $\partial\mathcal T$: the three vertices and the points $(x,0),(0,y),(x,1-x),(1-y,y),(x+y,0)$ and $(0,x+y)$. It is easy to prove that $\mu$ is continueous once $(x,y)$ approaches the boundary.

All this can be extended (with nontrivial effort) to allow higher-order Dirichlet boundary conditions (e.g.\ both $\mu$ and its normal  derivative), using cubic Hermite interpolation at the endpoints (i.e., points on $\partial\mathcal T$) -- cf.\ an example in \cite{huybrechs11fh5} in the case of an equilateral triangle.

\bibliographystyle{siamplain}

\end{document}

%% file: ex_shared.tex
\usepackage{bm}
\usepackage{dutchcal}
\usepackage{listings}
\usepackage{mathalfa}

\usepackage{mathtools}

\usepackage{arydshln}
\usepackage{dutchcal}
\usepackage{listings}
\usepackage{bm} % for \bm
\usepackage{fixmath} % for \mathbold

\usepackage{amsmath}
\usepackage{amsfonts}
\usepackage{mathrsfs}

\usepackage{lipsum}
\usepackage{amsfonts}
\usepackage{graphicx}
\usepackage{epstopdf}
\usepackage{algorithmic}
\ifpdf
  \DeclareGraphicsExtensions{.eps,.pdf,.png,.jpg}
\else
  \DeclareGraphicsExtensions{.eps}
\fi

\allowdisplaybreaks

\newcommand{\D}{{\hbox{d}}}
\providecommand{\MM}{\mathbcal}

\newcommand{\G}{\mathrm{\Gamma}}
\newcommand{\ee}{{\mathrm e}}
\newcommand{\ii}{{\mathrm i}}

\allowdisplaybreaks

\newsiamremark{remark}{Remark}
\newsiamremark{hypothesis}{Hypothesis}
\crefname{hypothesis}{Hypothesis}{Hypotheses}
\newsiamthm{claim}{Claim}

\headers{Spectral methods on a triangle and W-systems}{J. Gao and A. Iserles}

\title{Spectral methods on a triangle and W-systems}

\author{Jing Gao\thanks{The corresponding author, School of Mathematics and Statistics, Xi'an Jiaotong University 
  (\email{jgao@xjtu.edu.cn}\url{}).\funding{JG's work has been supported by National Natural Science Foundation of China (Grant No.\ 12271426), Natural Science Basic Research Plan in Shaanxi Province of China (Grant No.\ 2024JC-YBMS-040, \ 2024JC-JCQN-04), Shaanxi Fundamental Science Research Project for Mathematics and Physics under (Grant No.\ 23JSY025, \ 22JSY017).}}
  \and Arieh Iserles\thanks{Department of Applied Mathematics and Theoretical Physics, University of Cambridge(\email{ai@maths.cam.ac.uk}).}
  }

\usepackage{amsopn}